\documentclass{amsart}

\usepackage{xcolor,enumerate}
\definecolor{darkgreen}{rgb}{0,0.5,0}

\usepackage[utf8]{inputenc}
\usepackage[T1]{fontenc}
\usepackage{amsmath,amssymb,amsfonts,amsthm,graphicx}
\usepackage{txfonts}
\usepackage{helvet}
\usepackage{subcaption}
\usepackage{graphicx,tikz}
\usepackage{booktabs}
\usepackage{tkz-fct}
\usepackage{tkz-euclide}
\usepackage{natbib}
\usepackage{url}
\usepackage{xspace}
\usepackage{pifont}
\usepackage[most]{tcolorbox}
\usepackage[plain]{algorithm}
\usepackage[noend]{algorithmic}
\usepackage{threeparttable}
\usepackage{empheq}
\usepackage{hyperref}

\newcommand{\cmark}{\ding{51}}%
\newcommand{\xmark}{\ding{55}}%
\algsetup{indent=2em}

\newcommand\subscr[2]{#1_{\textup{#2}}}

\newcommand{\map}[3]{#1:#2 \rightarrow #3}

\newcommand\algobox[1]{
\begin{tcolorbox}[colback=white,  colframe=gray, boxrule=1pt, arc=2pt]
#1
\end{tcolorbox}
}

\newcommand{\fL}{\subscr{f}{L}}
\newcommand{\real}{\mathbb{R}}

\newcommand{\X}{\ensuremath{\mathcal{X}}}
\newcommand{\G}{\mathcal{G}}
\newcommand{\E}{\mathcal{E}}
\newcommand{\I}{\mathcal{I}}

\newcommand{\oprocendsymbol}{\hbox{$\bullet$}}
\newcommand{\oprocend}{\relax\ifmmode\else\unskip\hfill\fi\oprocendsymbol}

\DeclareMathOperator*{\argmax}{arg\,max}
\usetikzlibrary{external}
\usetikzlibrary{decorations.shapes,arrows,calc,decorations.markings,shapes,decorations.pathreplacing,shapes.arrows}
\usetikzlibrary{positioning}
\tikzset{main node/.style={circle,fill=blue!20,draw,inner sep=1pt},}
\tikzset{highlight1/.style={rectangle,
                           fill=red!15,
                           rounded corners = 0.5 mm,
                           inner sep=1pt,
                           fit=#1}}

\newcommand\hlight[1]{\tikz[overlay, remember picture,baseline=-\the\dimexpr\fontdimen22\textfont2\relax]\node[rectangle,fill=blue!13,rounded corners,fill opacity = 0.2,draw,thick,text opacity =1] {$#1$};}

\newtcbox{\mymath}[1][]{%
    nobeforeafter, math upper,
    enhanced, colframe=green!70!black,
    colback=white, boxrule=1pt,
    #1}

\newtheorem{theorem}{Theorem}[section]
\newtheorem{definition}[theorem]{Definition}

\newtheorem{remark}[theorem]{Remark}

\newtheorem{example}[theorem]{Example}

\newtheorem{proposition}[theorem]{Proposition}

\newcommand{\pdertwo}[2]{\frac{\partial^2 #1}{\partial #2^2}}
\numberwithin{equation}{section}

\begin{document}

\title[Submodular Optimization and Decision and Control]{Submodular Optimization with Applications to Decision and Control}

\author[Shamak Dutta]{Shamak Dutta}
\address{University of Waterloo, Department of Electrical and Computer Engineering, Waterloo, ON, N2L 3G1, Canada}
\email{shamak.dutta@uwaterloo.ca}

\author[Bahman Gharesifard]{Bahman Gharesifard}
\address{Queen's University, Department of Mathematics and Statistics, Kingston, ON, Canada}
\email{bahman.gharesifard@queensu.ca}

\author[Stephen L. Smith]{Stephen L.\ Smith}
\address{University of Waterloo, Department of Electrical and Computer Engineering, Waterloo, ON, N2L 3G1, Canada}
\email{stephen.smith@uwaterloo.ca}

\begin{abstract}
Submodular set functions, characterized by the diminishing-returns property, provide a unifying combinatorial framework for many subset-selection problems in decision and control. Although exact maximization is NP-hard in general, the structural properties of submodular functions enable simple greedy algorithms that achieve constant-factor approximation guarantees for monotone objectives, with randomized greedy-based variants extending such guarantees to the non-monotone case. This survey reviews the theory, algorithms, and applications of submodular optimization with a focus on systems and control. We cover the structural properties of submodular functions, including curvature and the submodularity ratio, the constraint families that arise in practice (matroids, knapsack, and $p$-systems), and the main approximation algorithms for monotone and non-monotone submodular maximization, with up-to-date approximation ratios and hardness results. We then survey applications across sensor scheduling, multi-agent coordination, robust submodular optimization, leader-follower systems, distributed submodular optimization, game theory, system theory, resource allocation, social networks, and informative path planning. The survey emphasizes practically implementable greedy-based algorithms and instance-dependent refinements via curvature and the submodularity ratio. We close with observations on canonical control-theoretic objectives: certain functionals are submodular (the log-determinant and rank of the controllability Gramian, and the log-determinant of the Kalman filter information matrix), whereas closely related objectives fail to be sub- or supermodular (the steady-state Kalman filter error covariance, and the average control energy obtained from the inverse Gramian). We also highlight the cross-cutting open directions that follow.
\end{abstract}


\maketitle


\section{Introduction}
Many problems in engineering, machine learning, and operations research require selecting a subset of resources, sensors, actuators, or decisions to optimize some system-level objective. When the number of candidate elements is large, exhaustive search is infeasible, and the combinatorial structure of the problem must be exploited to design efficient algorithms.

The main focus of this article is on combinatorial optimization problems of the form
    \begin{equation} \label{eqn:set-fn-opt}
        \begin{aligned}
            & \underset{S}{\mathrm{maximize}}
            & & f(S) \\
            & \text{subject to}
            & & S \in \I \subseteq 2^\X.
        \end{aligned}
    \end{equation}
Here, for a finite ground set $\X$, $f: 2^{\X} \to \real$ is the objective (a set function) that assigns values to subsets of $\X$ while $\I$ is a family of subsets of $\X$ representing the constraint of the problem. A set $S \subseteq \X$ is \emph{feasible} for \eqref{eqn:set-fn-opt} if $S \in \I$ and \emph{optimal} if $S$ is feasible, and $f(S) \geq f(T)$ for all $T \in \I$.
The number of feasible sets $|\I|$ is typically huge, as it grows exponentially in the size of the ground set, and in this sense, evaluating every feasible set in search of the optimal solution is not a viable option. Whether an efficient method to find the optimal solution, even approximately, can be developed depends on the structure of the objective function $f$ and the constraint set $\I$.

This survey studies optimization problems of the form in~\eqref{eqn:set-fn-opt} where $f$ is a \emph{submodular} function and $\mathcal{I}$ can represent a variety of combinatorial structures. Not only does this structure allow us to aim for efficient approximate algorithms, submodular optimization applies to many practical settings, in particular in \emph{decision and control}, a main focus of this work.
The central reason submodular functions are tractable is the \emph{diminishing returns} property: the marginal benefit of adding any element to a set decreases as the set grows larger. It turns out that this structure is enough to guarantee that simple greedy algorithms, adding one element at a time, achieve provably near-optimal solutions, often within a constant factor of the global optimum.

\emph{Scope and organization}: We focus primarily on maximization problems. For submodular minimization, we refer the reader to~\citet{bach2019submodular}, although we occasionally cite minimization-based results in the applications of Section~\ref{sec:applications} where they bear directly on problems in decision and control. Related surveys cover submodular optimization in machine learning~\citep{bilmes2022submodularity}, signal processing~\citep{tohidi2020submodularity}, the general theory of submodular maximization~\citep{AK-DG:12}, and, most closely related to ours, submodular maximization under uniform and partition matroids with distributed solutions for control and robotics~\citep{kia2025submodular}, together with the earlier monograph of~\citet{clark2016submodularity} on submodularity in the dynamics and control of networked systems; the present work is distinguished by its broader treatment of the constraint families and algorithms, its focus on \emph{systems and control} applications, its emphasis on practically implementable algorithms, and its inclusion of the most recently known approximation ratios and hardness results, which postdate that monograph. We also note that submodularity and matroid theory have had profound impact well beyond the algorithmic questions considered here. A striking recent example from pure mathematics is the development of combinatorial Hodge theory for matroids~\citep{adiprasito2018hodge}, among the contributions that led to the 2022 Fields Medal.

One of our aims is to provide a self-contained introduction to the main topics, enabling readers new to the area to quickly build the background needed to engage with current research. The survey is organized as follows: Section~\ref{sec:submod-functions} reviews the theory of submodular functions, including key examples, properties, curvature, and continuous extensions. Section~\ref{sec:constraints} discusses the main constraint families (matroids, knapsack, and $p$-systems). Section~\ref{sec:algorithms} presents the main approximation algorithms. Section~\ref{sec:applications} presents applications in decision and control and, when appropriate, highlights open directions.

\section{Submodular functions}\label{sec:submod-functions}
We begin with the definition of a submodular function.
\begin{definition}[Submodular function]\label{def:submod1}
Let $\X$ be a finite set.  Then a set function $f: 2^{\X} \to \real$ is \emph{submodular} if every pair of sets $S,T\subseteq \X$ satisfies
    \begin{equation}
        \label{ineq:submodularity}
        f(S) + f(T) \geq f(S\cup T) + f(S\cap T).
        \end{equation}
    \end{definition}
While this definition may seem abstract, it can be visualized intuitively in the following example.

\begin{example}[Example of submodularity]
{\em
Consider two overlapping sets in the plane as in Figure~\ref{fig:S_and_T}. 
\begin{figure}[hbt]
    \centering
    \includegraphics[width=0.3\linewidth]{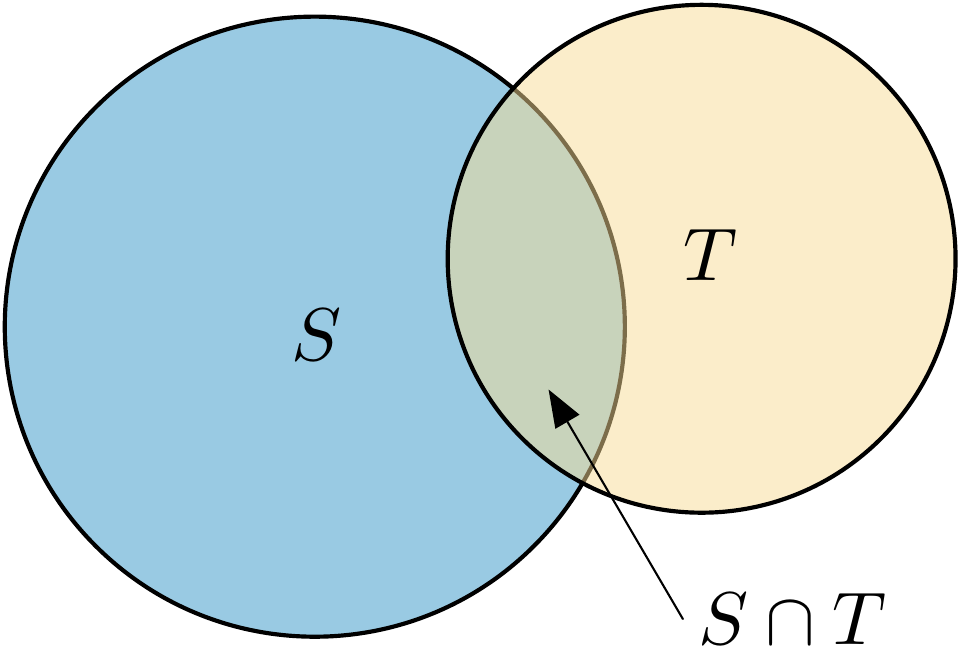}
    \caption{The area covered by a collection of (overlapping) sets is a submodular function.}
    \label{fig:S_and_T}
\end{figure}
Let $f$ be the set function that takes as input a collection of sets and outputs the area covered by the collection.  Here~\eqref{ineq:submodularity} asks that the sum of the covered areas $f(S) + f(T)$ be at least the area of their union plus that of their intersection, $f(S\cup T) + f(S\cap T)$.  For two regions in the plane this in fact holds with \emph{equality}---the \emph{inclusion-exclusion principle}, $f(S)+f(T)=f(S\cup T)+f(S\cap T)$---so the area measure is \emph{modular}. The function becomes strictly submodular once $f$ is the \emph{set-cover} function whose ground set is a collection of (overlapping) sets and whose value is the area of their union: each added set then contributes diminishing new coverage, making~\eqref{ineq:submodularity} a strict inequality whenever the sets overlap (we formalize this in Example~\ref{ex:set_cover_function}).  \oprocend
 }
\end{example}

We say that a set function $ f $ is \emph{supermodular} if $ -f$ is submodular, and \emph{modular} if it is both supermodular and submodular. In particular, for a given set $S\subseteq\X$, modular functions are simply represented by
\[
f(S)=f(\emptyset)+\sum_{s\in S} \big(f(s)-f(\emptyset)\big).
\]

We define the \emph{marginal gain} of adding the elements in $B\subseteq \X$ to a subset $A\subseteq \X$ as 
\[
\Delta(B|A)=f(A \cup B) - f(A).
\]
In the case where $B=\{x\}$ is a singleton, we write $\Delta(x|A)$ rather than $\Delta(\{x\}|A)$; analogously, we abbreviate $f(\{x\})$ as $f(x)$ for the value of $f$ on a singleton set.  The following is an equivalent characterization of submodularity, which will be used frequently. 
\begin{proposition}[An equivalent characterization of submodularity]\label{prop:equi-sub-char}
Let $ \X $ be a finite set. Then a set function $f: 2^{\X} \to \real$ is submodular if and only if every pair of sets $ S,T \subseteq \X $ with $ S\subseteq T$ and every $ x \in \X \setminus T $ satisfies
\[
f(S\cup \{x\}) - f(S) \geq f(T\cup \{x\}) - f(T), 
\]
which can be restated in terms of marginal gains as
\[
\Delta(x|S) \geq \Delta(x|T) 
\]
for all $S \subseteq T \subseteq \X$ and $x \in \X \setminus T$. 
\end{proposition}
\begin{proof}[Proof sketch]
Suppose that $f$ is submodular in the sense of~\eqref{ineq:submodularity}. Fix $S\subseteq T$ and $x\in \X\setminus T$. Applying~\eqref{ineq:submodularity} to the pair $T$ and $S\cup\{x\}$, note that $T\cup(S\cup\{x\})=T\cup\{x\}$ (since $S\subseteq T$) and $T\cap(S\cup\{x\})=S$ (since $x\notin T$). Therefore,~\eqref{ineq:submodularity} gives
\[
f(T)+f(S\cup\{x\}) \ge f(T\cup\{x\})+f(S),
\]
which rearranges to $f(S\cup\{x\})-f(S)\ge f(T\cup\{x\})-f(T)$.

Conversely, suppose the diminishing-returns property holds. For arbitrary $S,T\subseteq \X$, enumerate $S\setminus T=\{x_1,\ldots,x_r\}$ and define $A_i=(S\cap T)\cup\{x_1,\ldots,x_i\}$ and $B_i=T\cup\{x_1,\ldots,x_i\}$, so that $A_0=S\cap T$, $A_r=S$, $B_0=T$, and $B_r=S\cup T$. Since $A_{i-1}\subseteq B_{i-1}$ and $x_i\notin B_{i-1}$, diminishing returns gives 
\[
f(A_i)-f(A_{i-1})\ge f(B_i)-f(B_{i-1}). 
\]
Summing over $i=1,\ldots,r$ telescopes to $f(S)-f(S\cap T)\ge f(S\cup T)-f(T)$, which rearranges to~\eqref{ineq:submodularity}.
\end{proof}
This property captures the intuition that the marginal gain of adding an element decreases as the set grows larger.  As an example, consider the set of elements $\X$ to represent sensors (discs), and the submodular function $f$ to measure the area covered by the union of selected discs.  We can see in  Figure~\ref{fig:diminishing_returns} that the additional area covered by sensor $x$ when added to $A$ is more than when added to $B \supset A$.  
\begin{figure}
    \centering
    \includegraphics[width=0.38\linewidth]{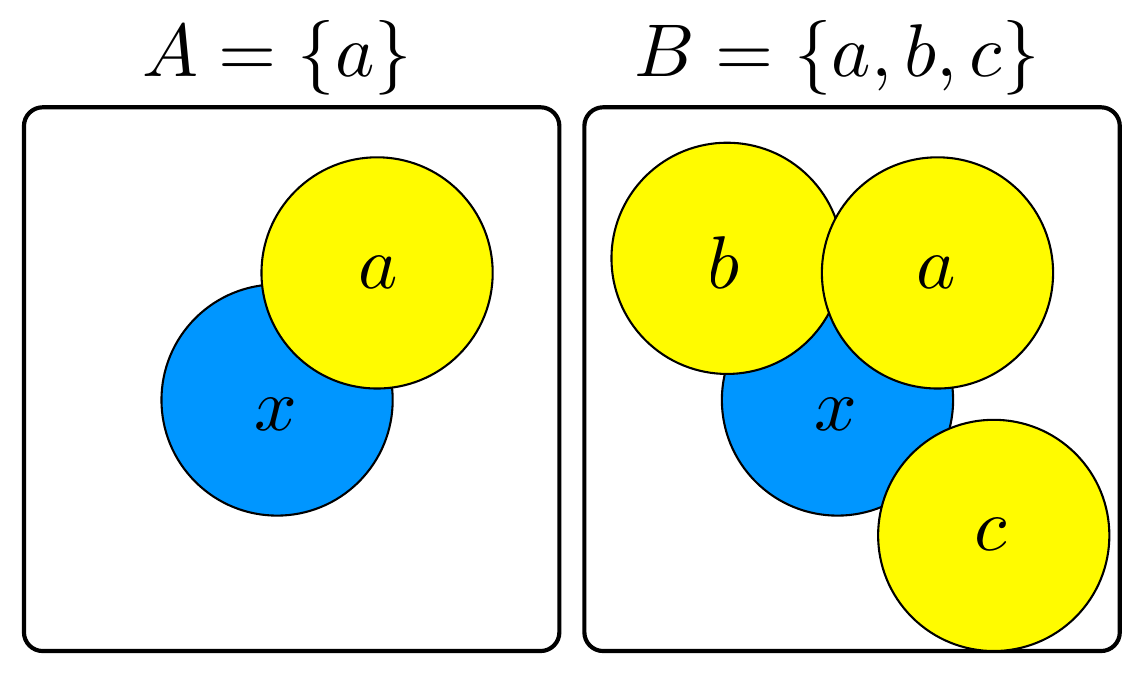}
    \caption{An example submodular function measuring the area covered by a set of discs, illustrating the notion of diminishing returns.}
    \label{fig:diminishing_returns}
\end{figure}

\subsection{Some key examples}

Beyond the set cover function already introduced, three further examples appear repeatedly in submodular optimization and in the applications discussed in this survey.
\begin{example}[Facility location]
{\em Let $\mathcal{U}$ be a set of clients and $\X$ a set of potential facility locations, with $w_{ij}\geq 0$ denoting the benefit of assigning client $j\in \mathcal{U}$ to facility $i \in \X$. The \emph{facility location} function
\[
    f(A) = \sum_{j\in\mathcal{U}} \max_{i\in A}\, w_{ij}, \qquad A\subseteq\X,
\]
measures the total benefit when each client is served by its best open facility (with the convention that $\max_{i\in\emptyset}(\cdot)=0$, i.e., $f(\emptyset)=0$). Submodularity follows from the diminishing-returns property of the max: for $A\subseteq B$ and $x\notin B$, client $j$'s gain satisfies
\[
\max(w_{xj}-\max_{i\in B}w_{ij},0)\leq \max(w_{xj}-\max_{i\in A}w_{ij},0),
\]
since $\max_{i\in B}w_{ij}\geq\max_{i\in A}w_{ij}$. Adding more facilities can only maintain or improve service, and $f(\emptyset)=0$. Facility location is widely used in clustering, recommendation, and document summarization.}\oprocend
\end{example}

\begin{example}[Entropy and mutual information]
{\em Let $\mathbf{X}=(X_1,\ldots,X_n)$ be a collection of jointly Gaussian random variables with covariance matrix $\Sigma$.
For any index set $S\subseteq\{1,\ldots,n\}$, let $\mathbf{X}_S:=(X_i)_{i\in S}$ denote the random vector obtained by stacking the coordinates of $\mathbf{X}$ with indices in $S$ (in increasing order). For a singleton $\{x\}$ we write $X_x:=X_{\{x\}}$.
The \emph{differential entropy} of $\mathbf{X}_A$ is
\[
    H(A)=\tfrac12\log\det(\Sigma_{AA})+c_{|A|},
\]
where $\Sigma_{AA}$ is the principal submatrix of $\Sigma$ indexed by $A$, and $c_{|A|}$ depends only on $|A|$ (in particular, not on $\Sigma$).
For $A\subseteq B$ and $x\notin B$, let the conditional entropy be denoted
\[
H(X_x\mid \mathbf{X}_A)=H(A\cup\{x\})-H(A).
\]
For random vectors $U,V,W$ (of arbitrary dimensions), define the conditional mutual information by
\[
I(U;V\mid W):=H(U\mid W)-H(U\mid V,W),
\]
where $H(U\mid V,W)$ means conditional differential entropy given the joint random vector $(V,W)$ (equivalently, the stacked vector $[V^\top\,W^\top]^\top$).
This function is submodular because
\[
    H(X_x\mid \mathbf{X}_A) - H(X_x\mid \mathbf{X}_B)
    = I \bigl(X_x\,;\,\mathbf{X}_{B\setminus A}\mid \mathbf{X}_A\bigr)\ge 0,
\]
i.e., conditioning on more observations can only reduce (or maintain) conditional entropy, so marginal entropy gains are diminishing. We have $H(\emptyset)=0$ (up to the constant). Note that, unlike its discrete counterpart, this differential entropy is submodular but not in general monotone: the marginal gain $H(X_x\mid \mathbf{X}_A)$ is a conditional differential entropy, which can be negative when $X_x$ is nearly determined by $\mathbf{X}_A$. This example arises directly in sensor placement~\citep{JMLR:v9:krause08a}: since the total entropy 
\[
H(\{1,\ldots,n\})=H(A)+H(\{1,\ldots,n\}\setminus A\mid \mathbf{X}_A)
\]
is fixed, selecting $k$ indices $A$ to maximize $H(A)$ is equivalent to minimizing the residual uncertainty $H(\{1,\ldots,n\}\setminus A\mid \mathbf{X}_A)$ in the unselected variables.
\oprocend}
\end{example}
\begin{example}[Actuator selection in a network system]
{\em Consider a network of $n$ subsystems described by the discrete-time linear dynamics 
\[
x_{t+1} = A x_t + \sum_{i \in S} b_i u_i(t),
\]
where $A \in \real^{n \times n}$ captures the network coupling, and each column $b_i \in \real^n$ represents the influence of actuator $i$ on the state. For a horizon $T$, the \emph{controllability Gramian} of the selected actuator set $S \subseteq \X$ is
\[
    W(S) = \sum_{t=0}^{T-1} A^t B_S B_S^\top (A^\top)^t,
\]
where $B_S = [b_i]_{i \in S}$ is the matrix whose columns are the $ b_i $, $ i \in S $. It is well-known that the spectral properties of the controllability Gramian measure the required control energy~\citep{pasqualetti2014controllability}. For instance, a natural measure of how well the selected actuators can steer the system is
\[
f(S)=\tfrac12\log\det(I+W(S)),
\]
which is, up to an additive constant, the differential entropy of a Gaussian reachable-state model with covariance $I+W(S)$, and hence serves as a scalar proxy for overall controllability. Interestingly, this function is submodular: for $S \subseteq T$ and $x \notin T$, the additional log-determinant gain from adding actuator $x$ to the larger set $T$ is no greater than adding it to $S$, because the directions already covered by $T \supseteq S$ leave less residual gain for $x$ to contribute.}\oprocend
\end{example}

\subsection{Properties}

We start by stating some elementary set function properties that are commonly assumed for submodular functions. 
Let $ \X $ be a finite set, and let $f: 2^{\X} \rightarrow \real$  be a set function. 
\begin{definition}[Properties of set functions]
    A set function $f:2^{\X}\rightarrow \real$ is 
\begin{enumerate}
    \item \emph{non-negative} if for any subset $A$ of $\X$, $f(A) \geq 0$. This property ensures that the value of a subset is always non-negative;

    \item \emph{monotone} if $f(A) \leq f(B)$ whenever $A \subseteq B$, for all $A, B \subseteq \X$. In other words, adding elements to a set cannot decrease its value.

    \item \emph{normalized} if $f(\emptyset) =0$.  That is, the empty set has a value of zero.

    \item \emph{symmetric} if for every set $A\subseteq \X$, we have $f(A) = f(\X\setminus A)$.  This property is typically imposed on non-monotone functions; for symmetric submodular functions in particular (such as the graph cut function or $|A|\,|\X\setminus A|$), the function values tend to be largest for sets containing roughly half the elements and smallest at the extremes (the empty and full sets).
\end{enumerate}
\end{definition}
When a set function is monotone, normalized, and submodular, it is sometimes called a \emph{polymatroid set function}~\citep{edmonds2003submodular}; this class of functions plays a central role in the analysis of greedy algorithms (see~\citet{GLN-LAW-MLF:78-I}, and also Section~\ref{sec:algorithms}). The following example verifies that the set cover function satisfies all three properties.

\begin{example}[Properties of the set cover function]
\label{ex:set_cover_function}
{\em Consider a set of $n$ elements $\mathcal{U} = \{e_1,\ldots,e_n\}$ and a collection of $m$ sets $\{S_1,\ldots,S_m\}$, where each $S_i \subseteq \mathcal{U}$.  Let $\mathcal{X} =  \{1,\ldots,m\}$ be the indices of the sets, and let $f:2^{\mathcal{X}} \to \real_{\geq 0}$ be the set cover function, which for each $A\in 2^{\mathcal{X}}$ returns the number of elements in $\mathcal{U}$ covered by $A$:
\[
f(A) = \left|\bigcup_{i\in A} S_i\right|.
\]
Note that $A\subseteq \mathcal{X}$ is a set of indices of sets from $\{S_1,\ldots,S_m\}$.

The function $f$ is normalized since $f(\emptyset) = |\emptyset| = 0$.  The function is monotone since for any $A\subset B \subseteq \mathcal{X}$ we have
\[
f(B) = f(A) + \Delta(B \mid A),
\]
and 
\[
\Delta(B|A) = \left|\big(\cup_{i\in B}S_i\big) \setminus \big( \cup_{i\in A}S_i \big) \right| \geq 0.
\]
Finally, to show submodularity, take any $A\subseteq B \subset \mathcal{X}$ and $x\notin B$.  We need to establish that $\Delta(x|A) \geq \Delta(x|B)$.  We have
\begin{align*}
\Delta(x|B) &= \left| S_x \setminus \cup_{i\in B}S_i \right|, \\
    &= \left| S_x \setminus \left( \left(\cup_{i\in A}S_i\right) \cup \left(\cup_{i\in B\setminus A}S_i\right) \right) \right| \\
    &= \left| \left( S_x\setminus \cup_{i\in A}S_i \right) \cap \left( S_x\setminus \cup_{i\in B\setminus A}S_i \right) \right| \\
    & \leq \left|  S_x\setminus \cup_{i\in A}S_i \right| = \Delta(x|A).
\end{align*}
where the equality in the third line is an application of De Morgan's laws: $A \setminus (B \cup C) = (A \setminus B) \cap (A\setminus C)$.  Thus, the set cover function is normalized, monotone and submodular. Note that normalized and monotone implies non-negativity.  The function is clearly not symmetric as $f(\emptyset) = 0 \neq f(\X)$ unless $\X$ is the empty set.\oprocend
}
\end{example}

We next define the notion of curvature, an instrumental tool for quantifying how far a monotone submodular function deviates from modularity, which in turn yields refined performance guarantees for greedy and related approximation algorithms.
\begin{definition}[Curvature]\label{def:curvature}
The \emph{curvature} of a normalized, monotone submodular function $f: 2^{\X} \rightarrow \real_{\geq 0}$ is defined as
\begin{equation}\label{eq:curvature}
    c = 1-\min_{\substack{A\subseteq \X,\, x\in \X\setminus A \\ f(x)>0}}\frac{\Delta(x|A)}{f(x)}.
\end{equation}
\end{definition}
By diminishing returns, the minimum in~\eqref{eq:curvature} is attained at $A=\X\setminus\{x\}$, recovering the standard form $c=1-\min_{x:\,f(x)>0}\,\Delta(x\mid \X\setminus\{x\})/f(x)$ of~\citet{CONFORTI1984251}. Whenever $c = 0$, the function is modular; in this sense, the notion of curvature measures how far $f$ is from being modular. Curvature is an instance-dependent structural parameter, and as we will see in Section~\ref{sec:algorithms} it sharpens the worst-case guarantees of the greedy algorithm.

A complementary parameter applies to set functions that are not exactly submodular. The \emph{submodularity ratio} of a non-negative monotone set function $f:2^\X \to \real_{\geq 0}$ is
\[
\gamma \;=\; \min_{\substack{S,\, T \subseteq \X,\; S \cap T = \emptyset \\ \Delta(T\mid S)>0}}\;
\frac{\sum_{x \in T} \Delta(x\mid S)}{\Delta(T\mid S)},
\]
with $\gamma \in [0,1]$, and $\gamma = 1$ if and only if $f$ is submodular~\citep{das2011submodular}. A function with $\gamma$ bounded away from zero is said to be \emph{weakly submodular}. Like curvature, $\gamma$ refines the greedy guarantee (see Section~\ref{sec:algorithms}), and it is the standard tool used in this survey when an objective in decision and control is not exactly submodular. Several other instance-dependent variants tailored to specific applications, such as extended greedy curvature~\citep{welikala2022new} and supermodularity of conditioning~\citep{corah2018distributed}, appear in Section~\ref{sec:applications}, where they yield application-specific refinements.

\subsection{Operations that preserve submodularity}\label{sec:submod-operations}

A key practical reason submodularity is easy to work with is that the class of submodular functions is closed under several natural operations. This means complex objectives encountered in applications can often be verified as submodular by decomposing them into simpler pieces. We outline some of these next. 
\begin{enumerate}
        \item \emph{The weighted sum of submodular functions is submodular.}
             That is, given submodular functions $f_i:2^{\X}\rightarrow \real$, for $i\in\{1,\ldots,n\}$ and scalars $\alpha_1,\alpha_2,\ldots,\alpha_n \geq 0$, the function $g:2^{\X}\rightarrow \real$ defined as $g(S) = \sum_{i=1}^n\alpha_i f_i(S)$ for each $S\subseteq \X$ is submodular. This is useful whenever one wants to optimize a multi-objective tradeoff: for example, a weighted combination of coverage and sensing quality is submodular if each term individually is.
            
        \item \emph{The composition of a non-decreasing concave function and a monotone submodular function is submodular.}  That is, given a non-decreasing concave function $\varphi:\real \to \real$ and a monotone submodular function $f:2^{\X}\rightarrow \real$, the function $(\varphi\circ f):2^{\X}\rightarrow \real$ defined as $(\varphi\circ f)(S) = \varphi\big(f(S)\big)$ for each $S\subseteq \X$ is submodular. A canonical instance is the set-cover function of Example~\ref{ex:set_cover_function}, which is monotone submodular: composing it with a non-decreasing concave function (such as $\sqrt{\cdot}$ or $\min(\cdot,k)$) yields a submodular objective that captures the diminishing utility of additional coverage. The log-determinant entropy $H$ of the previous section is also submodular, but by the conditional-entropy argument given there (conditioning cannot increase entropy) rather than by this composition rule: $\det(\Sigma_{AA})$ is not itself a monotone submodular set function, so the rule does not apply to it.
        
        \item \emph{The residual of $f$ given a set $B$ is submodular.}  That is, given a submodular function $f:2^{\X}\rightarrow \real$ and a set $B\subset \X$, the function $g:2^{\X\setminus B}\rightarrow \real$ defined as \[
        g(S) = f(S \cup B) - f(B)
        \]
        for each $S \subseteq \X\setminus B$ is submodular. Practically, this means that after any partial solution $B$ is fixed, e.g., after several greedy steps, the remaining marginal value is still a submodular function over the remaining elements, a property that underpins the analysis of greedy algorithms.
    \end{enumerate}

\subsection{Continuous extensions}
Submodular functions admit several natural ``extensions'' to continuous domains, which are useful both for analysis and for designing algorithms. We briefly describe the two most common; readers primarily interested in the algorithmic results may skip this subsection.

\subsubsection*{Multi-linear extension} 
The multilinear extension $F: [0,1]^n \to \mathbb{R} $ of a submodular function $f:2^\X\to\real$, where we identify $\X=\{1,\ldots,n\}$ so that $n=|\X|$, is defined as:
\[
F(\mathbf{x}) = \sum_{S \subseteq \X} f(S) \prod_{i \in S} x_i \prod_{j \in \X \setminus S} (1 - x_j),
\]
where $ \mathbf{x} \in [0,1]^n $. There is another way to make sense of this function, by observing that 
\[
F(\mathbf{x}) = \mathbb{E}[f(S)],
\]
where the expectation is taken over a random set $S \subseteq \X$ such that each element $i \in \X$ is included in $S$ independently with probability $x_i$, the $i$th entry of the vector $\mathbf{x}$.
The multilinear extension provides us with a differentiable function on a continuous domain, whose gradient can be estimated efficiently by sampling. It underlies the continuous greedy algorithm of~\citet{calinescu2011maximizing}, which achieves the optimal $(1-1/e)$ approximation ratio for monotone submodular maximization subject to a matroid constraint (Section~\ref{sec:algorithms}), though at substantially higher computational cost than the algorithms discussed in this survey.


\subsubsection*{Lov\'{a}sz extension}

A submodular function can be extended to a convex function, commonly referred to as the \emph{Lov\'{a}sz extension}, over a polyhedron. To wit, consider a submodular function $ \map{f}{2^\X}{\real} $, where $ \X=\{1,2,\cdots, n\} $. Given a set $ A \subseteq \X $, and a vector $ v \in \real^n $, we define 
\[
v(A):=\sum_{i\in A} v_i.
\]
We now define the submodular polyhedron
$ P_f $ associated to $ f $ as
\[
P_f=\{v\in \real^n \ | \ v(A) \leq f(A) \quad \mathrm{for} \ \mathrm{all} \ A \subseteq \X \}.
\]
The  \emph{Lov\'{a}sz extension} of $ f $ is a function $ \map{\fL}{[0,1]^n}{\real} $ given by
\[
\fL(w)=\max_{v\in P_f} w^\top v.
\]
In spite of the fact that the constraint set for this linear program can be complicated, its solution can easily be represented in closed form. Given $ w \in [0,1]^n $, consider a permutation $\sigma$ such that $w_{\sigma(1)} \geq w_{\sigma(2)} \geq w_{\sigma(3)} \geq \cdots \geq w_{\sigma(n)}$.
Let us define $v^* \in \mathbb{R}^n$ using its components, indexed by $ \sigma(i) $, where $ i \in\{1,\cdots, n\} $ as
   \[
   v^*_{\sigma(i)} = f(A_{\sigma(i)}) - f(A_{\sigma(i-1)}),
   \]
   where $A_{\sigma(0)}:=\emptyset$ and
   \[
   A_{\sigma(i)}=\{\sigma(1),\cdots, \sigma(i)\} \quad \text{for } i\geq 1.
   \]
Note that by construction, $v^* \in P_f$: for any subset $A \subseteq \X$, the constraint $v^*(A) \leq f(A)$ follows from the submodularity of $f$ (see, e.g., \citet{edmonds2003submodular,bach2013learning,SF:05}). Moreover,
   \[
   \fL(w) = \sum_{i=1}^n w_{\sigma(i)}\, v^*_{\sigma(i)} = \sum_{i=1}^n w_{\sigma(i)} \bigl(f(A_{\sigma(i)}) - f(A_{\sigma(i-1)})\bigr),
   \]
and a standard exchange argument on the chain of tight constraints $v^*(A_{\sigma(i)}) = f(A_{\sigma(i)})$ shows that $v^*$ maximizes $w^\top v$ over $v\in P_f$ (the so-called \emph{greedy algorithm for the submodular polyhedron}; see \citet{edmonds2003submodular}). Hence
\begin{align*}
    \fL(w) &= \max \left\{ \sum_{i=1}^n w_i v_i : v \in P_f \right\} \\
    &= \sum_{i=1}^n w_{\sigma(i)} \left( f(A_{\sigma(i)}) - f(A_{\sigma(i-1)}) \right),
\end{align*}
where $ \sigma $ is a permutation of the indices of $ w $ such that the components of $ w $ are sorted in non-increasing order.

The key property of this extension is established in the seminal paper of Lov\'{a}sz, which states that minimizing a submodular function is equivalent to minimizing its Lov\'{a}sz extension. 
\begin{theorem}[Lov\'{a}sz~\citep{LL:83}]\label{thm:lovasz}
Let $ f $ be a submodular function defined on $ \X=\{1\ldots, n\} $ such that $ f(\emptyset)=0 $. Then
\begin{enumerate}
\item $ \map{\fL}{[0,1]^n}{\real} $ is a convex function;
\item $\min\{f(A) \mid A\subseteq \X \}=\min\{\fL(w) \mid w\in [0,1]^n \}$, and the continuous minimization can be solved in polynomial time\footnote{For instance, one can use the ellipsoid method; \citet{schrijver2000combinatorial} later gave a strongly-polynomial combinatorial algorithm for submodular function minimization.}.
\end{enumerate}
\end{theorem}
 We refer to~\citet{bach2013learning} for a detailed treatment of the submodular minimization problem and its applications.

\subsubsection*{DR-submodularity}\label{sec:DR-sub}
Continuous submodular functions~\citep{bach2019submodular,bian2017continuous} extend the discrete framework to real-valued domains and have found applications in systems and control~\citep{bunton2022joint}. We briefly define this extension next: Let 
$ \mathcal{X}=\Pi_{i=1}^n \mathcal{X}_i$, where each $ \mathcal{X}_i=[0,a_i]$, where $ i \in \{1,\ldots, n\}$, is an interval in $ \real $. We consider the partial order on $ \real^n $ where $ x\leq y $ if and only if $ x_i \leq y_i $. A function $ \map{f}{\mathcal{X}}{\real_{\geq 0}} $ is said to be \emph{submodular} if for all $x,y \in \mathcal{X}$ we have that
\[
f(x)+f(y)\geq f(x \wedge y )+ f(x \vee y ),
\]
where $x \wedge y = \min(x,y)$ (componentwise minimum) and $x \vee y = \max(x,y)$ (componentwise maximum). Moreover, $ f $ is said to be monotone, if it is monotone with respect to this partial order, i.e., $ x\geq y $ implies  $f(x) \geq f(y) $, for all $ x,y \in \mathcal{X}$. One can verify that whenever $ f $ is twice differentiable, it is submodular if and only if 
\[
\frac{\partial^2f(x)}{\partial x_i\partial x_j}\leq 0,
\]
for all $ x\in \mathcal{X} $ and all $ i \neq j $. It is worth pointing out that this condition is different from convexity and concavity, even though there are convex and concave functions that satisfy it.  A \emph{strict subclass} of continuous submodular functions~\citep{bian2017guaranteed} is called DR-submodular. While we do not define this precisely here, we point out that whenever $ f $ is twice differentiable, this condition additionally asks that
$\pdertwo{f(x)}{x_i}\leq 0
$ for all $ i \in \{1,\ldots, n\}$. 
In~\citet{hassani2017gradient}, it is shown that stochastic projected gradient methods can be used to approximate the solution of a submodular maximization problem within a $ \frac{1}{2}$-factor for maximizing any monotone continuous DR-submodular function over a general convex constraint set. Prior to this result, for instance in~\citet{bian2017guaranteed}, it is shown that under the down-closed convex constraint\footnote{A convex set $C\subseteq\mathcal{X}$ (with lower bound $x_0$) is \emph{down-closed} if $x\in C$ and $x_0\leq z \leq x$ imply $z\in C$.}, using a version of the Frank-Wolfe algorithm, the approximation factor of $ (1-\frac{1}{e}) $ can be obtained, and this approximation is tight.
At the time of writing this article, this area is cutting-edge with a lot of recent activities; for the current state of the art, we refer the reader to~\citet{mualem2024bridging} and references within. 

\subsection{Representation of submodular functions}

Since we are dealing with the optimization of submodular functions, this raises the question of how a submodular function is represented. One approach is to list values for all possible subsets of the ground set $\X$. However, this takes space that is exponential in the size of the ground set (since we need to enumerate all possible subsets). There are two ways around this challenge, which we discuss next.

        \subsubsection*{Value oracle}
        The value oracle model assumes the function $f$ is given as a black box. The algorithm can only query the value of $f(S)$ for any given subset S of the ground set. The hardness of a problem is then measured by the minimum number of oracle queries needed to find an approximate or exact solution. For example, a problem is considered hard if it requires an exponential number of queries. This model is useful for proving lower bounds that are independent of the specific representation of the function.
        
        \subsubsection*{Compact representation}
        When the function $f$ has a compact representation, it means the function is described by an explicitly given data structure or formula. For example, in the case of a submodular function, the function might be represented by a cut function in a graph, where the input is the graph itself. The size of this representation is typically polynomial in the size of the ground set. Hardness results in this model consider the running time of an algorithm as a function of the size of this representation. For example, problems are hard if they require exponential time even when the function is given in a compact form. This model is more powerful since the algorithm can use the structure of the representation to its advantage, rather than just querying values.

The choice of oracle model determines which complexity bounds apply. In Section~\ref{sec:algorithms}, all algorithm complexities are stated in terms of the number of value-oracle calls, which is the standard model for the approximation algorithms we discuss.

\section{Set system constraints}\label{sec:constraints}

    \subsection{Independence systems and matroids}
    Many practical combinatorial optimization settings involve constraints. Since submodular functions are defined on subsets of the set $\X$, such constraints can conveniently be thought of as restricting what sets the submodular function has access to. Given the characterization of a submodular function provided in Proposition~\ref{prop:equi-sub-char}, for instance, we need to place some restrictions on the type of constraints that we allow. This motivates the next definition. 
    
    \begin{definition}[Independence systems and matroids]\label{def:independent-set}
        Let $ \X $ be a finite set. We say that $ (\X, \I) $ is an \emph{independence system} if $\I \subseteq 2^{\X}$ satisfies:
        \begin{enumerate}
            \item  $\emptyset \in \I $;
            \item for every $ A\in \I$, we have that $ B \in \I $ for all $ B \subseteq A$ (\emph{downward-closed property}).
        \end{enumerate}
        We say that $ A $ is an \emph{independent set} if $ A\in \I$, and \emph{dependent} otherwise. An independence system that additionally satisfies:
        \begin{enumerate}
            \setcounter{enumi}{2}
            \item if $ A,B\in \I$ and $ |A|>|B|$, then there exists $ a\in A\backslash B $ such that $ \{a\}\cup B \in \I $ (\emph{augmentation property})
        \end{enumerate}
        is called a \emph{matroid}, denoted $M=(\X,\I)$.
    \end{definition}

    \subsection{Examples of Matroids} 
    Matroid theory in full generality is beyond the scope of this survey; we refer the reader to~\citet{oxley2006matroid} for a detailed treatment. We nevertheless include a few illustrative examples. Our first example shows that the concept of matroid generalizes the notion of linear independence for vector spaces. 
    \begin{example}[Matroid over a finite subset of a vector space]
    {\em Let $ V $ be a vector space and let $ \X $ be a finite subset of $ V $. Then the collection $ \I $ of linearly independent subsets of $ \X $ forms an independence system, with a corresponding matroid $ M=(\X,\I) $. }\oprocend
    \end{example}
    
    To provide an example, consider an undirected graph $ \G=(V,\E)$, with vertex set $ V$ and edge set $ \E$. We say that $ \G $ is \emph{acyclic} (or \emph{forest}) if it does not have any cycle, and say that $ \G$ is a \emph{tree} if it is additionally connected.

    \begin{example}[Graphic matroid]
    {\em Let $ \G=(V,\E)$ be an undirected graph, and let $ \I \subseteq 2^{\E} $ be the set of all subsets of edges of $ \G $ that form a forest (see Figure~\ref{fig:ex-graph-tree}). 
    \begin{figure}[htb!]
    \centering
    \begin{tikzpicture}[scale=0.95,node distance=1.4cm]
        \node[main node] (1) {$1$};
        \node[main node] (2) [below right of=1]  {$2$};
                    \node[main node] (7) [right of=2]  {$7$};
        \node[main node] (3) [below left of= 1] {$3$};
        \node[main node] (4) [below of= 2] {$4$};
        \node[main node] (5) [below of= 3] {$5$}; 
            \node[main node] (6) [ left of= 5] {$6$}; 
        \path[draw]
        (7) edge (1)
        (1) edge (3)
        (1) edge (2)
        (2) edge (3)
        (2) edge (4)
        (3) edge (5)
        (5) edge(6);    
        \end{tikzpicture}
    \hspace{4em}
    \begin{tikzpicture}[scale=0.95,node distance=1.4cm]
        \node[main node] (1) {$1$};
        \node[main node] (2) [below right of=1]  {$2$};
                    \node[main node] (7) [right of=2]  {$7$};
        \node[main node] (3) [below left of= 1] {$3$};
        \node[main node] (4) [below of= 2] {$4$};
        \node[main node] (5) [below of= 3] {$5$}; 
            \node[main node] (6) [ left of= 5] {$6$}; 
        \path[draw]
        (7) edge (1)
        (1) edge (2)
        (2) edge (4)
        (3) edge (5)
        (5) edge(6);    
        \end{tikzpicture}
    \caption{An undirected graph (left) and a subgraph that is a forest (right). Any subgraph of the graph on the right is still a forest.}\label{fig:ex-graph-tree}
    \end{figure}
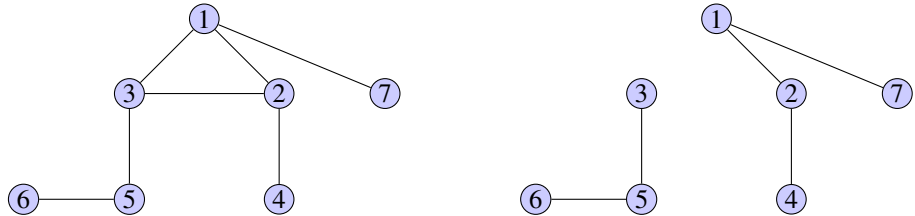
    Then $ \I $ is an independence system: suppose that $ G' $ is a subgraph formed by a set of edges $ B' \in \I $, and consider $ A' \subseteq B' $. Then the subgraph constructed by the edges in $ A' $ is a subgraph of the forest $ G' $, hence itself acyclic, so $A' \in \I $. Moreover, the augmentation property holds, because a forest with smaller number of edges than another can be grown to a bigger one by using an edge of the larger forest.} 
    \oprocend
    \end{example}

    \begin{example}[Partition matroid]
        \em{Let $ \X $ be a finite set partitioned into $ k $ non-empty subsets $ \X_1,\ldots, \X_k $, with a set of positive integers $ \ell_1,\ldots, \ell_k$ associated to each subset respectively. Then it is easy to verify that the collection $ \I $ of all subsets $ A\subseteq \X$ such that $ | A \cap \X_i |\leq \ell_i $ for all $ i \in \{1,\ldots, k\} $ forms an independence system. The corresponding matroid $ M=(\X,\I)$ is called the partition matroid and plays a key role in some of our constructions.} \oprocend
    \end{example}

\begin{example}[Uniform matroid]
{\em Perhaps the simplest and most commonly encountered matroid is the \emph{uniform matroid} $U_{k,n}$ on a ground set $\X$ with $|\X|=n$. Its independent sets are precisely the subsets of $\X$ of size at most $k$:
\[
    \I = \{A \subseteq \X : |A| \leq k\}.
\]
This corresponds to a \emph{cardinality constraint}: one may select at most $k$ elements. The augmentation property is immediate: if $A,B\in\I$ with $|A|<|B|$, then $|A|<k$, and for any $x\in B\setminus A$ we have $A\cup\{x\}\in\I$. The uniform matroid is the structure associated with cardinality-constrained submodular maximization; see Theorem~\ref{thm:greedy-cardinality} below.}\oprocend
\end{example}
    \subsection{Knapsack}

    In certain combinatorial optimization problems, each element of the ground set $e \in \X$ has an associated weight $w_e \geq 0$. A knapsack constraint enforces a budget $B > 0$ on the total selected weight i.e., for a set $S \subseteq \X$, the constraint is
    \begin{equation}
        \sum_{e \in S} w_e \leq B.
    \end{equation}
    The family $\mathcal{I} = \{S \subseteq \X : \sum_{e \in S} w_e \leq B\}$ is an \emph{independence system}: $\emptyset \in \mathcal{I}$ (trivially, since the empty sum is zero), and for any $A \in \mathcal{I}$ and $A' \subseteq A$, we have $\sum_{e \in A'} w_e \leq \sum_{e \in A} w_e \leq B$, so $A' \in \mathcal{I}$ (downward-closed property). However, $\mathcal{I}$ is generally \emph{not} a matroid because the augmentation property can fail. To see this, let $\X = \{a, b, c\}$ with weights $w_a = 2, w_b = 1, w_c = 1$ and budget $B = 2$. Then $A = \{a\} \in \mathcal{I}$ (weight 2) and $A' = \{b, c\} \in \mathcal{I}$ (weight 2) are both feasible with $|A| = 1 < |A'| = 2$. Yet adding either $b$ or $c$ to $A$ gives total weight $3 > B$, so no augmentation from $A'$ into $A$ is possible, violating the augmentation property. Knapsack constraints generalize the cardinality constraint (recovered when all weights are equal) but, as the example shows, lack matroid structure, and therefore require a dedicated algorithm rather than the plain greedy; we present a greedy-based $(1-1/\mathrm{e})$-approximation for the knapsack-constrained problem in Section~\ref{sec:algorithms} (Theorem~\ref{thm:greedy-knapsack}).

    \subsection{Intersection of matroids and $p$-systems}

    Many practical constraints arise as the \emph{intersection} of $k$ matroids on the same ground set $\X$: a set $S$ is feasible if and only if $S \in \mathcal{I}_1 \cap \cdots \cap \mathcal{I}_k$. For example, requiring a set to be a common independent set of two matroids (e.g., satisfying both a cardinality and a partition constraint simultaneously) yields a $k=2$ case. The intersection of $k$ matroids is itself an independence system (downward-closed) though not generally a matroid. This structure is the setting of the greedy algorithm for $k$-matroid intersections analyzed in Section~\ref{sec:algorithms}, which achieves a $\tfrac{1}{k+1}$-approximation ratio for monotone submodular maximization; the approximation degrades gracefully as $k$ increases. More generally, this independence system is a $p$-system with $p=k$, and the same $1/(p+1)$-guarantee extends to this broader class~\citep{MLF-GLN-LAW:78-II}. Formally, $(\X,\I)$ is a \emph{$p$-system} if, for every $Y\subseteq\X$, any two inclusion-wise maximal independent subsets of $Y$ have sizes within a factor $p$ of each other.

\section{Algorithms}\label{sec:algorithms}

    \subsection{Submodular optimization problems and results}
    Recall that we are interested in problems of the form
    \begin{equation} 
        \begin{aligned}
            & \underset{S}{\mathrm{maximize}}
            & & f(S) \\
            & \text{subject to}
            & & S \in \I \subseteq 2^\X
        \end{aligned}
    \end{equation}
    where $f$ is a submodular set function and $\I$ is a family of subsets of the finite ground set $\X$ that represents the constraints. This problem (even when unconstrained) is NP-hard as it contains the \textsc{Max-Cut} problem, known to be NP-hard \citep{karp1975computational}, as a special case. Thus, efforts have focused on the development of \emph{approximation algorithms} where we try to find a solution that closely approximates the optimal value of the problem.
    
    Throughout this section, we assume access to a \emph{function evaluation oracle} for $f$ (returning $f(S)$ for any queried $S$) and an \emph{independence oracle} for $\I$ (returning whether $S \in \I$ for any queried $S$). Algorithm complexity is measured in the number of such oracle calls.

    \begin{definition}[Approximation algorithm]
        An $\alpha$-approximation algorithm for an optimization problem is a polynomial-time algorithm that, for every instance of the problem, produces a solution whose value is within a factor of $\alpha$ of the value of an optimal solution. The value $\alpha$ is known as the approximation factor, and for maximization problems we always have $0 < \alpha \leq 1$.
    \end{definition}

    We restrict our study to non-negative functions since we are looking for multiplicative approximation factors. 

    A variety of algorithms have been developed for submodular maximization. We summarize the main results in submodular maximization in Table~\ref{tab:submod-results}. For each case, we give the algorithm with the best known approximation factor along with the best attainable approximation factor (under standard complexity-theoretic assumptions). The approximation factors of these algorithms depend on whether the function $f$ is monotone as well as on the structure of the constraint set $\I$. The algorithms can be broadly classified into greedy, local search, and relax and round. Local search and relax-and-round methods have the best approximation factors but have high oracle complexities and are primarily of theoretical interest. In contrast, the greedy algorithm is fast, simple to implement, and in some cases achieves the best approximation factor (sometimes in combination with random subsampling).
    
    In this section, we describe the greedy algorithm along with its variants (acceleration and subsampling) and analyze worst-case performance.
    
        \begin{table*}[]
        \renewcommand{\arraystretch}{1.25}
            \centering
            \footnotesize
            \resizebox{\textwidth}{!}{%
            \begin{tabular}{@{}c  c  c c c c  c@{}}
            \toprule
                    \multicolumn{2}{c}{Problem} & \phantom{a} & Hardness & \phantom{a} & \multicolumn{2}{c}{Results}\\
                    \cmidrule{1-2} \cmidrule{6-7}
                    Monotone & Set System & && & Algorithm Type & Approximation ratio \\ 
                    \midrule
                    \cmark & Cardinality && $1-\frac{1}{\mathrm{e}}$ \citep{nemhauser1978best} && \textbf{Greedy} \citep{GLN-LAW-MLF:78-I} & $1-\frac{1}{\mathrm{e}}$ \\[1em]
                    
                    \cmark & $1$ matroid && $1-\frac{1}{\mathrm{e}}$ \citep{vondrak2013symmetry} && \textbf{Continuous greedy} \citep{calinescu2011maximizing} & $1-\frac{1}{\mathrm{e}}$  \\
                    &&&&& Greedy \citep{MLF-GLN-LAW:78-II} & $\frac{1}{2}$ \\[1em]
                    
                    \cmark & $k$ matroids && $\frac{\log{k}}{k}$ \citep{hazan2006complexity} \footnote{The authors in \citet{hazan2006complexity} show that the maximum $k$-dimensional matching problem cannot be approximated to a factor better than $O(\log k/k)$. This also holds for submodular maximization over $k$ matroids because the $k$-dimensional matching problem is a special case of maximizing a linear function over a partition matroid.} && \textbf{Local search} \citep{lee2010submodular} & $\frac{1}{k+\epsilon}$  \\
                    &&&&& Greedy \citep{MLF-GLN-LAW:78-II} & $\frac{1}{k+1}$ \\[1em]
                    
                    \xmark &  Unconstrained && $\frac{1}{2}$ \citep{feige2011maximizing} && \textbf{Double greedy (randomized)} \citep{buchbinder2015tight} & $\frac{1}{2}$ \\
                    &&&&& Double greedy (deterministic) \citep{buchbinder2015tight} & $\frac{1}{3}$ \\
                    &&&&& Uniformly random set \citep{feige2011maximizing} & $\frac{1}{4}$ \\[1em]
                    
                    \xmark &  Cardinality && 0.478 \citep{qi2024maximizing} && \textbf{Continuous greedy} \citep{buchbinder2024constrained} & 0.401 \\
                    &&&&& Random greedy \citep{buchbinder2014submodular} & $1/\mathrm{e}$ \\[1em]
                    
                    \xmark &  1 matroid && 0.478 \citep{qi2024maximizing}  && \textbf{Continuous greedy} \citep{buchbinder2024constrained} & 0.401 \\
                    &&&&& Sample greedy \citep{harshaw2022power} & 0.25 \\[1em]
                    
                    \xmark &  $k$ matroids && $\frac{\log{k}}{k}$ \citep{hazan2006complexity} && \textbf{Local search} \citep{lee2010submodular} & $\frac{1}{k+1+\frac{1}{k-1}+\epsilon}$\\
                    &&&&& Sample greedy \citep{harshaw2022power} & $\frac{k}{(k+1)^2}$\\ 
                    \bottomrule
            \end{tabular}
            }
            \caption{Known results for variations of the centralized submodular maximization problem. The algorithms in bold achieve the best known approximation guarantees. Non-bolded entries achieve weaker approximation factors but are often simpler or more scalable; greedy-based algorithms, in particular, remain the standard choice for large-scale problems. Local search and continuous-greedy methods match the best known factors but typically at substantially higher computational cost; this survey focuses on the greedy-based algorithms.}
            \label{tab:submod-results}
        \end{table*}

\subsection{Greedy algorithm}

In the absence of additional information, it is natural to study the performance of a greedy strategy which at every iteration aims at maximizing the marginal reward. In particular, we start with the empty $ S=\emptyset$ and at iteration $ t\geq 1 $ of the greedy procedure, we grow the set by adding an element, if it exists, with maximal marginal return such that the resulting set stays feasible with respect to the constraint set $ \I $. This is formally displayed as $\texttt{Greedy}(f,\X,\mathcal{I})$. 

\begin{center}
\begin{minipage}{8cm}
\begin{algorithm}[H]
\algobox{
\textbf{function} $\texttt{Greedy}(f,\mathcal{X},\mathcal{I})$
\begin{algorithmic}[1]
\REQUIRE{A submodular function $f$, and a set system $(\X, \I)$}
\ENSURE{A set $S \in \mathcal{I}$ that greedily maximizes $f(S)$}
  \STATE $S = \emptyset$
  \REPEAT 
  \STATE Let $\mathcal{X}_{\text{valid}} = \{x \in\mathcal{X}\setminus{S} \;|\; S\cup\{x\} \in \mathcal{I}\}$
  \IF{$\mathcal{X}_{\text{valid}} = \emptyset$} 
    \RETURN $S$
  \ENDIF
  \STATE $x \in \argmax_{\bar x \in \mathcal{X}_{\text{valid}}} \Delta(\bar x\;|\;S)$
  \STATE $S := S \cup \{x\}$
  \UNTIL{}
\end{algorithmic}
}
\end{algorithm}
\end{minipage}
\end{center}

An example execution of the greedy algorithm is shown in Figure~\ref{fig:greedy_cardinality}, where base set $\X$ is the set of five discs shown for Iteration 0,  the submodular function $f$ is the area covered by the selected sensors, and the set $\I$ is the cardinality constraint $\I = \big\{S \subseteq \X \; \big|\; |S| \leq 3\big\}$. 

The approximation factor of the greedy algorithm depends on whether the function $f$ is monotone and on the type of set system $\I$. We begin with arguably the simplest setting of a monotone function subject to a cardinality constraint.

\begin{figure}[htb]
    \centering
     \begin{subfigure}[b]{0.2\textwidth}
         \centering
         \includegraphics[width=\textwidth]{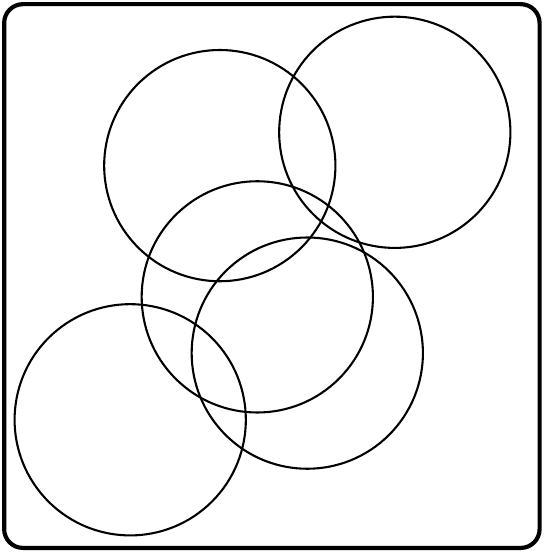}
         \caption{Iteration 0}
         \label{fig:greedy1}
     \end{subfigure}
     \begin{subfigure}[b]{0.2\textwidth}
         \centering
         \includegraphics[width=\textwidth]{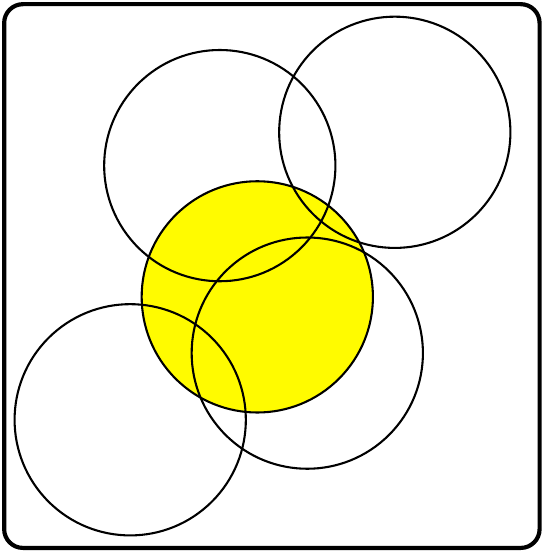}
         \caption{Iteration 1}
         \label{fig:greedy2}
     \end{subfigure}
     \begin{subfigure}[b]{0.2\textwidth}
         \centering
         \includegraphics[width=\textwidth]{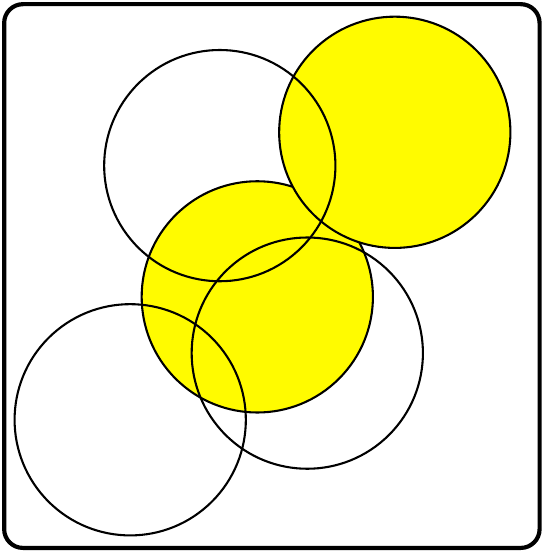}
         \caption{Iteration 2}
         \label{fig:greedy3}
     \end{subfigure}
     \begin{subfigure}[b]{0.2\textwidth}
         \centering
         \includegraphics[width=\textwidth]{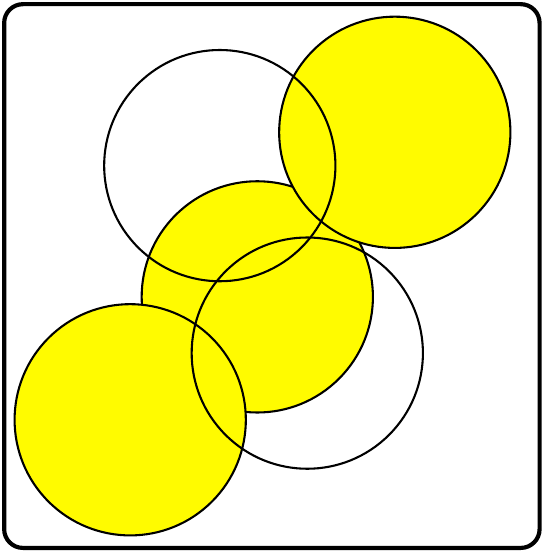}
         \caption{Iteration 3}
         \label{fig:greedy4}
     \end{subfigure}
    \caption{An example execution of the greedy algorithm on a sensor coverage problem with the cardinality constraint $|S| \leq 3$.}
    \label{fig:greedy_cardinality}
\end{figure}

\subsubsection*{Monotone objective with a cardinality constraint}

One of the most well-known results in submodular maximization is the analysis of the greedy algorithm for a monotone submodular function subject to a cardinality constraint \citep{GLN-LAW-MLF:78-I}.

\begin{theorem}[Performance of Greedy Algorithm for Cardinality Constraint~\citep{GLN-LAW-MLF:78-I}]\label{thm:greedy-cardinality}
Consider the cardinality-constrained submodular maximization problem with a normalized, monotone submodular function $f:2^{\X} \to\real$, subject to $|S|\leq k$ for a given positive integer $k$, and let $S^* \in \argmax_{|S|\leq k} f(S)$. Then the greedy algorithm returns a set $S_G$ satisfying
\[
\frac{f(S_G)}{f(S^*)} \geq \left(1 - \frac{1}{\mathrm{e}}\right).
\]
\end{theorem}
\begin{proof}
Consider the sequence of sets $S_{t+1} := S_t \cup \{x\}$, where 
\[
x \in \argmax_{\bar x \in \mathcal{X}_{\text{valid}}} \Delta(\bar x\;|\;S_t)
\]
generated by the greedy algorithm, where $S_0 = \emptyset$. We claim that 
\begin{equation}\label{eq:proof-1e-aux1}
f(S_{t}) - f(S_{t-1}) \geq
\frac{1}{k}\,\bigl(f(S^*) - f(S_{t-1})\bigr).    
\end{equation}
Let us first show that when this claim holds, the proof easily follows: note that by~\eqref{eq:proof-1e-aux1}, we have that
\[
f(S_t) 
\geq \left(1 - \frac{1}{k}\right) f(S_{t-1}) + \frac{1}{k}\,f(S^*).
\]
By using this inequality iteratively, we have that 
\[
f(S_t) 
\geq 
f(S^*)\Bigl(1 - \bigl(1 - \tfrac{1}{k}\bigr)^t\Bigr).
\]
If the greedy algorithm exhausts the feasible choices before step $k$ (possible only when $|\X|\le k$), then it has selected every element, so $S_G=\X$ and $f(S_G)\ge f(S^*)$ by monotonicity. Otherwise it performs $k$ steps; setting $t = k$ and noting that 
$(1 - \tfrac{1}{k})^k \le e^{-1}$, giving
\[
f(S_G) 
\geq
f(S^*)\Bigl(1 - \bigl(1 - \tfrac{1}{k}\bigr)^k\Bigr) 
\;\ge\; 
f(S^*)\Bigl(1 - \tfrac{1}{e}\Bigr).
\]
Hence $f(S_G) / f(S^*) \ge 1 - 1/e$, as required. It remains to show that~\eqref{eq:proof-1e-aux1} is true: let us denote $R_{t-1} = S^* \setminus S_{t-1}$. By monotonicity, we have that
\[
f(S_{t-1} \cup R_{t-1}) \ge f(S^*).
\]
Using this along with submodularity, we have that 
\begin{align*}
f(S^*)- f(S_{t-1})  &\leq
f(S_{t-1} \cup R_{t-1}) - f(S_{t-1}) \\
&\leq \sum_{x \in R_{t-1}} \bigl(f(S_{t-1} \cup \{x\}) - f(S_{t-1})\bigr)    \\
&\leq |R_{t-1}|\left(
f(S_t) - f(S_{t-1})\right) \\   
&\leq k\left(
f(S_t) - f(S_{t-1})\right),    
\end{align*}
which proves~\eqref{eq:proof-1e-aux1}, hence finishing the proof. 
\end{proof}

The power of this result is that the simple greedy algorithm (which runs in polynomial time) is guaranteed to find a solution whose value is at least $(1-1/\mathrm{e}) \approx 63\%$ of the optimal. This worst-case bound admits two instance-dependent refinements via the structural parameters introduced in Section~\ref{sec:submod-functions}. Using the curvature $c$ of $f$ (Definition~\ref{def:curvature}), the guarantee improves to $\tfrac{1}{c}(1-e^{-c})$~\citep{CONFORTI1984251}, which recovers $1 - 1/e$ at $c = 1$ and collapses to $1$ when $f$ is modular ($c = 0$); the optimal curvature-dependent factor of $1 - c/\mathrm{e}$, together with a matching hardness bound, was later established by~\citet{sviridenko2017optimal}. When $f$ is not exactly submodular but has submodularity ratio~$\gamma$, an analogue of the proof above yields the guarantee $(1 - e^{-\gamma})$ for weakly submodular maximization with a cardinality constraint~\citep{das2011submodular}, recovering $1 - 1/e$ at $\gamma = 1$.

\subsubsection*{Monotone objective with matroid constraints}
The worst-case guarantee of the greedy algorithm for matroids has also been studied \citep{MLF-GLN-LAW:78-II}.

\begin{theorem}[Performance of Greedy Algorithm for the Intersection of $k$ Matroids~\citep{MLF-GLN-LAW:78-II}]\label{thm:greedy-matroid}
Consider the submodular maximization problem with a normalized, monotone submodular function $f:2^{\X} \to\real$ subject to the intersection of $k$ matroids with $k \geq 1$, and let $S^*$ be an optimal solution. Then the greedy algorithm returns a set $S_G$ satisfying
\[
\frac{f(S_G)}{f(S^*)} \geq \frac{1}{k+1}.
\]
When $k=1$, the greedy solution $f(S_G)$ is at least $1/2$ of the optimal value $f(S^*)$.
\end{theorem}
We refer to~\citet{MLF-GLN-LAW:78-II} for the proof. As in the cardinality case, the worst-case factor admits a curvature refinement: for a single matroid constraint, the $\tfrac{1}{2}$ guarantee sharpens to $\tfrac{1}{1+c}$~\citep{CONFORTI1984251}, again recovering $1/2$ at $c=1$ and collapsing to $1$ in the modular case $c=0$.

\subsubsection*{Monotone objective with a knapsack constraint}
A knapsack constraint (Section~\ref{sec:constraints}) generalizes the cardinality constraint by assigning each element $e$ a weight $w_e \geq 0$ and bounding the total weight of the selected set by a budget $B$. Here the plain greedy algorithm can perform arbitrarily badly, since adding the element of largest marginal gain ignores its cost: a single heavy element can exhaust the budget while contributing little relative to many lighter ones. A natural fix is the \emph{cost-benefit greedy}, which adds at each step the feasible element of largest marginal-gain-to-weight ratio $\Delta(x\mid S)/w_x$; on its own, this still does not guarantee a constant factor in the worst case. A simple remedy, due to~\citet{khuller1999budgeted} for budgeted maximum coverage, is to return the better of the cost-benefit greedy solution and the best single feasible element, which already yields a $\tfrac{1}{2}(1-1/\mathrm{e})$ guarantee. \citet{sviridenko2004note} showed that running cost-benefit greedy from every feasible seed set of size at most three and returning the best of these runs---a \emph{partial enumeration}---recovers the optimal guarantee.

\begin{theorem}[Greedy for a knapsack constraint~\citep{sviridenko2004note}]\label{thm:greedy-knapsack}
Let $f:2^{\X}\to\real_{\geq 0}$ be a normalized, monotone submodular function, and consider the knapsack-constrained problem $\max\{f(S) : \sum_{e\in S} w_e \leq B\}$ with optimal solution $S^*$. The partial-enumeration cost-benefit greedy algorithm returns a set $S$ satisfying
\[
\frac{f(S)}{f(S^*)} \geq 1 - \frac{1}{\mathrm{e}}.
\]
\end{theorem}

This $(1-1/\mathrm{e})$ factor coincides with the cardinality case, which a knapsack with unit weights recovers, and is the best attainable for monotone submodular maximization. The size-three enumeration is what lifts cost-benefit greedy to the optimal factor, at the price of an $O(n^5)$ oracle complexity; in practice the cheaper cost-benefit greedy, optionally with lazy evaluation, is often used as a faster heuristic.

\subsubsection*{Double Greedy}
The (deterministic) double greedy algorithm of~\citet{buchbinder2015tight} solves the unconstrained submodular maximization problem (USM) with a $\frac{1}{3}$-approximation guarantee; its randomized variant ($\texttt{RandomizedUSM}$, below) attains the tight $\frac{1}{2}$. The algorithm maintains two sets: $X$ (initially empty) and $Y$ (initially the full ground set $\X$). For each element $u_i$, it computes two quantities: $a_i = f(X_{i-1} \cup \{u_i\}) - f(X_{i-1})$, the gain from adding $u_i$ to $X$, and $b_i = f(Y_{i-1} \setminus \{u_i\}) - f(Y_{i-1})$, the loss from removing $u_i$ from $Y$. The element is then greedily assigned: if adding $u_i$ to $X$ is more beneficial, set $X_i = X_{i-1} \cup \{u_i\}$; otherwise, remove $u_i$ from $Y$. At termination, by construction $X_n = Y_n$, and this common set is returned.

\begin{center}
\begin{minipage}{9cm}
\begin{algorithm}[H]
\algobox{
\textbf{function} $\texttt{DoubleGreedy}(f, \X)$
\begin{algorithmic}[1]
\REQUIRE{A submodular function $f$, ground set $\X = \{u_1, \dots, u_n\}$ in a fixed order}
\ENSURE{A set $X_n \subseteq \X$}
\STATE {$X_0 \leftarrow \emptyset$, $Y_0 \leftarrow \X$}
\FOR{$i \leftarrow 1$ \textbf{to} $n$}
  \STATE $a_i \leftarrow f(X_{i-1} \cup \{u_i\}) - f(X_{i-1})$
  \STATE $b_i \leftarrow f(Y_{i-1} \setminus \{u_i\}) - f(Y_{i-1})$
  \IF{$a_i \geq b_i$}
    \STATE $X_i \leftarrow X_{i-1} \cup \{u_i\}$, $Y_i \leftarrow Y_{i-1}$
  \ELSE
    \STATE $X_i \leftarrow X_{i-1}$, $Y_i \leftarrow Y_{i-1} \setminus \{u_i\}$
  \ENDIF
\ENDFOR
\RETURN $X_n$
\end{algorithmic}
}
\end{algorithm}
\end{minipage}
\end{center}

\begin{theorem}[Performance of Double Greedy~\citep{buchbinder2015tight}]\label{thm:double-greedy}
Let $f : 2^{\X} \to \real_{\geq 0}$ be a non-negative submodular function and let $S^* \subseteq \X$ be an optimal solution to the unconstrained maximization problem. Then $\texttt{DoubleGreedy}$ returns a set $X_n$ satisfying
\[
\frac{f(X_n)}{f(S^*)} \geq \frac{1}{3}.
\]
This is the best factor attainable by the deterministic algorithm; randomizing the assignment (Theorem~\ref{thm:randomized-usm}) improves it to the tight $\frac{1}{2}$.
\end{theorem}

\subsubsection*{Lazy Greedy}
$\texttt{LazyGreedy}$ \citep{minoux1978accelerated}, later popularized as CELF~\citep{leskovec2007cost}, is an efficient implementation of the greedy algorithm for submodular maximization. It achieves significant practical speedups by leveraging submodularity to avoid redundant function evaluations, which are often the primary computational bottleneck.

Consider the cardinality constrained problem. This involves selecting a set $S$ of size at most $k$ from a ground set $V$ that maximizes a submodular function $f(S)$. The standard greedy algorithm solves this iteratively: in each of the $k$ steps, it evaluates the marginal gain $\Delta(e|S) = f(S \cup \{e\}) - f(S)$ for every element $e \notin S$ and adds the one with the highest gain to $S$. This requires $O(nk)$ function evaluations, which can be computationally prohibitive for large datasets.

The diminishing returns property of submodularity implies that a marginal gain calculated in an earlier iteration of the greedy algorithm (with a smaller set $S$) serves as a valid {upper bound} on the element's true marginal gain in any subsequent iteration. $\texttt{LazyGreedy}$ uses these upper bounds to lazily defer re-computation, only evaluating an element's true marginal gain when it becomes a candidate for selection. The algorithm maintains a max-priority queue that stores all candidate elements, prioritized by their marginal gains. The key is that these stored gains can be outdated. Note that the worst case runtime is still $O(nk)$ which matches the vanilla greedy algorithm but offers significant speedups in practice.

\begin{remark}
With consistent tie-breaking, $\texttt{LazyGreedy}$ returns exactly the same set as the standard greedy algorithm and therefore inherits the same approximation guarantees. The speedup is purely in the number of function evaluations: in practice it is often orders of magnitude faster than $\texttt{Greedy}$ (see~\citet{minoux1978accelerated}).
\end{remark}

\begin{center}
\begin{minipage}{8.8cm}
\begin{algorithm}[H]
\label{alg:lazy-greedy}
\footnotesize
\algobox{
\textbf{function} $\texttt{LazyGreedy}(f, \X, k)$
\begin{algorithmic}[1]
\REQUIRE{A submodular function $f$, a ground set $\X$, an integer $k$}
\ENSURE{A set $S \subseteq \X$ with $|S|=k$ that greedily maximizes $f(S)$}
  \STATE $S \leftarrow \emptyset$
  \STATE Initialize a max-priority queue $Q$
  \FOR{each element $x \in \X$}
    \STATE $\delta_x \leftarrow f(x)$ 
    \STATE $Q.\text{push}(x, \delta_x)$
  \ENDFOR
  \FOR{$i \leftarrow 1$ \textbf{to} $k$}
    \LOOP
      \STATE $(x, \delta_x) \leftarrow Q.\text{pop}()$ 
      \STATE $\delta'_{x} \leftarrow \Delta(x\;|\;S)$ 
      \IF{$Q$ is empty}
        \STATE $S \leftarrow S \cup \{x\}$
        \STATE \textbf{break} 
      \ENDIF
      \STATE $(x_{next}, \delta_{next}) \leftarrow Q.\text{peek}()$ 
      \IF{$\delta'_{x} \ge \delta_{next}$}
        \STATE $S \leftarrow S \cup \{x\}$ 
        \STATE \textbf{break} 
      \ELSE
        \STATE $Q.\text{push}(x, \delta'_{x})$ 
      \ENDIF
    \ENDLOOP
  \ENDFOR
  \RETURN $S$
\end{algorithmic}
}
\end{algorithm}
\end{minipage}
\end{center}

\subsubsection*{Non-monotone objective with cardinality constraints}
The greedy algorithm is often used for the non-monotone case as well. However, unlike the monotone setting, it provides no meaningful worst-case guarantee, even when implemented with an early stopping rule (halting as soon as all remaining marginal gains are non-positive). The following star-graph example demonstrates this.

\begin{example}[Greedy fails for non-monotone submodular maximization]
{\em  Let $N = \{c, \ell_1, \ldots, \ell_{n-1}\}$ consist of a center $c$ and $n-1$ leaves $L = \{\ell_1, \ldots, \ell_{n-1}\}$. Consider the cardinality-constrained problem with budget $k = n-1$ and define the non-negative, non-monotone function
\[
f(S) =
\begin{cases}
|S \cap L| & \text{if } c \notin S, \\
1 + \epsilon & \text{if } c \in S,
\end{cases}
\]
for small $\epsilon > 0$. This function is submodular: for $A \subseteq B$ and $x \notin B$, the only non-trivial case is $c \notin B$, $x = c$, where $\Delta(c \mid A) = 1 + \epsilon - |A \cap L| \geq 1 + \epsilon - |B \cap L| = \Delta(c \mid B)$. Greedy with early stopping proceeds as:
\begin{itemize}
\item \textbf{Step 1.} $\Delta(c \mid \emptyset) = 1+\epsilon > 1 = \Delta(\ell_i \mid \emptyset)$, so greedy selects $c$.
\item \textbf{After $c$:} $\Delta(\ell_i \mid \{c\}) = (1+\epsilon) - (1+\epsilon) = 0$ for every leaf. All marginals are zero, so early-stopping greedy halts.
\end{itemize}
Greedy returns $S_G = \{c\}$ with $f(S_G) = 1+\epsilon$. The optimal feasible solution is $S^* = L$ with $f(S^*) = n-1$, giving
\[
\frac{f(S_G)}{f(S^*)} = \frac{1+\epsilon}{n-1} \;\xrightarrow{\epsilon\to 0}\; \frac{1}{n-1}.
\]
Crucially, no implementation choice rescues greedy: greedy picks $c$ over any leaf in Step~1 (since $1+\epsilon > 1$), and thereafter all marginals are zero so no further improvement is possible. The star-graph intuition is that the center $c$ ``saturates'' the function, making all leaves redundant once $c$ is selected, even though the leaves collectively have value $n-1 \gg 1 + \epsilon$.} \oprocend
\end{example}

The root cause of greedy's failure is its deterministic nature: once a high-gain element like $c$ is selected, the algorithm is locked into a bad partial solution and no subsequent element can recover value. The $\texttt{RandomGreedy}$ algorithm of~\citet{buchbinder2014submodular} breaks this commitment by \emph{randomizing} the selection at each step. Rather than always taking the single best element, it identifies the top-$k$ marginal-gain candidates and picks one \emph{uniformly at random}, with a carefully chosen probability of skipping the step entirely. This randomization prevents the algorithm from being trapped by any single high-gain decoy, since the optimal solution's elements have a non-negligible chance of being selected at each step regardless of which element would have been chosen greedily.

\begin{center}
\begin{minipage}{9cm}
\begin{algorithm}[H]
\algobox{
\textbf{function} $\texttt{RandomGreedy}(f, \X, k)$
\begin{algorithmic}[1]
\REQUIRE{A submodular function $f$, a ground set $\X$, an integer $k$}
\ENSURE{A set $A_k \subseteq \X$}
  \STATE $A_0 \leftarrow \emptyset$
  \FOR{$i \leftarrow 1$ \textbf{to} $k$}
    \STATE $M_i \leftarrow$ the set of $\min(k,\,|\X \setminus A_{i-1}|)$ elements of $\X \setminus A_{i-1}$ with highest marginal gains $\Delta(u \mid A_{i-1})$
    \STATE \textbf{with} probability $(1 - |M_i|/k)$: $A_i \leftarrow A_{i-1}$
    \STATE \textbf{otherwise} Let $u_i$ be a uniformly random element of $M_i$, and set $A_i \leftarrow A_{i-1} \cup \{u_i\}$
  \ENDFOR
  \RETURN $A_k$
\end{algorithmic}
}
\end{algorithm}
\end{minipage}
\end{center}

\begin{theorem}[Performance of Random Greedy~\citep{buchbinder2014submodular}]\label{thm:random-greedy}
Let $f : 2^{\X} \to \real_{\geq 0}$ be a non-negative submodular function (not necessarily monotone), let $S^*$ be an optimal solution to the cardinality-constrained problem $\max_{|S| \leq k} f(S)$, and let $A_k$ be the output of $\texttt{RandomGreedy}$. Then
\[
\mathbb{E}\bigl[f(A_k)\bigr] \;\geq\; \frac{1}{\mathrm{e}}\, f(S^*).
\]
\end{theorem}
This $O(nk)$-time algorithm is the simplest result in~\citet{buchbinder2014submodular}; the same paper gives more refined algorithms for the cardinality-constrained case that improve upon $1/\mathrm{e}$, achieving $1/\mathrm{e} + 0.004$ in general and $1/2 - o(1)$ when $k = n/2$. Note that the $1/\mathrm{e}$ bound here applies to the cardinality-constrained non-monotone setting; the unrelated $1/4$ bound for a uniformly random set in Table~\ref{tab:submod-results} applies to the unconstrained setting.


\subsection{Stochastic Greedy for Subsampling}

The algorithms in the previous section evaluate the submodular function over a large number of candidate elements at each step, which becomes prohibitive when $n$ is large. A powerful and recurring idea is \emph{subsampling}: rather than querying $f$ on every candidate, evaluate on a small random subset. Submodularity's diminishing-returns structure makes this safe; a random subsample of the ground set retains a representative portion of the marginal gain information, so the greedily selected element from the sample is nearly as good as the true greedy choice. This idea surfaces in several forms: $\texttt{StochasticGreedy}$ subsamples at each greedy step to eliminate the dependence of the runtime on $k$; the random uniform set provides an $O(n)$ baseline for unconstrained problems; and $\texttt{SampleGreedy}$ uses a random half of $\X$ as a pre-filter before running standard greedy. In each case, randomization trades a small loss in approximation factor for a significant reduction in oracle complexity.


For very large scale applications, running $\texttt{LazyGreedy}$ \citep{minoux1978accelerated} is impractical. The question of whether it can be accelerated further to remove the dependence of the runtime on $k$ (the number of elements to be selected) was resolved in \citep{mirzasoleiman2015lazier}. The authors developed $\texttt{StochasticGreedy}$: an algorithm that uses $O(n \log(1/\epsilon))$ function evaluations with a guarantee of $1-1/e - \epsilon$ for monotone submodular maximization subject to a cardinality constraint. Note that this number of evaluations does not depend on $k$. The core idea is to reduce the size of the ground set (thereby reducing the number of function evaluations) during each greedy step by randomly sampling a subset of the ground set. By carefully selecting the size of the subset (based on $\epsilon$), the authors show that this results in only $\epsilon$ loss in the approximation guarantee. A closely related thresholding scheme that also achieves near-linear running time was introduced earlier by~\citet{badanidiyuru2014fast}.

\begin{center}
\begin{minipage}{9cm}
\begin{algorithm}[H]
\algobox{
\textbf{function} $\texttt{StochasticGreedy}(f, \X, k, \epsilon)$
\begin{algorithmic}[1]
\REQUIRE{A submodular function $f$, a ground set $\X$, an integer $k$, and $\epsilon > 0$}
\ENSURE{A set $S \subseteq \X$ with $|S|=k$ that greedily maximizes $f(S)$}
  \STATE $S \leftarrow \emptyset$
  \FOR{$i \leftarrow 1$ \textbf{to} $k$}
      \STATE $R \leftarrow$ \text{a uniformly random subset of}\ $\X \setminus S$ \text{of size}\ $\min\!\bigl((n/k)\log(1/\epsilon),\, |\X\setminus S|\bigr)$
      \STATE $x \leftarrow \argmax_{\bar{x} \in R} \Delta(\bar{x}|S)$
      \STATE $S := S \cup \{x\}$
  \ENDFOR
  \RETURN $S$
\end{algorithmic}
}
\end{algorithm}
\end{minipage}
\end{center}
\begin{theorem}[Performance of Stochastic Greedy~\citep{mirzasoleiman2015lazier}]\label{thm:stochastic-greedy}
Let $f : 2^{\X} \to \real_{\geq 0}$ be a normalized, monotone submodular function and let $S^*$ be an optimal solution to the cardinality-constrained problem $\max_{|S| \leq k} f(S)$. For any $\epsilon > 0$, $\texttt{StochasticGreedy}$ returns a set $S$ satisfying
\[
\mathbb{E}[f(S)] \;\geq\; \left(1 - \frac{1}{\mathrm{e}} - \epsilon\right) f(S^*)
\]
in $O(n\log(1/\epsilon))$ oracle calls, independent of $k$.
\end{theorem}

\subsection{Randomized Algorithms for Unconstrained Maximization}

For unconstrained non-monotone submodular maximization, there is no feasibility constraint on the output set, but the non-monotonicity means marginal gains can be negative, leaving no clear greedy direction. Randomizing the include/exclude decision for each element sidesteps this difficulty and yields provable constant-factor guarantees; we discuss this next.

\subsubsection*{Random Uniform Set}
A simple baseline for unconstrained non-monotone submodular maximization is the \emph{random uniform set}: independently include each element $u \in \X$ with probability $1/2$. Despite its simplicity, this already provides a constant-factor guarantee.

\begin{theorem}[Random Uniform Set~\citep{feige2011maximizing}]\label{thm:random-uniform}
Let $f : 2^{\X} \to \real_{\geq 0}$ be a non-negative submodular function and let $S$ be a uniformly random subset of $\X$ (each element included independently with probability $1/2$). Then
\[
\mathbb{E}[f(S)] \;\geq\; \frac{1}{4}\, f(S^*),
\]
where $S^*$ is any fixed subset of $\X$.
\end{theorem}

The $1/4$ bound is intuitive: by submodularity, the expected value of a random half of the ground set is at least $1/4$ of the optimum. 

\subsubsection*{RandomizedUSM} 
The \texttt{RandomizedUSM} algorithm of~\citet{buchbinder2015tight} achieves the tight $1/2$ guarantee by running a \emph{double-sided} random process: it simultaneously grows a set $X$ from $\emptyset$ and shrinks a set $Y$ from $\X$, probabilistically including each element based on the balance between marginal gains $a_i$ (gain to $X$) and marginal losses $b_i$ (loss from $Y$).

\begin{center}
\begin{minipage}{9cm}
\begin{algorithm}[H]
\algobox{
\textbf{function} $\texttt{RandomizedUSM}(f, \X)$
\begin{algorithmic}[1]
\REQUIRE{A submodular function $f$, ground set $\X=\{u_1, \dots, u_n\}$ in a fixed order}
\ENSURE{A set $X_n \subseteq \X$}
\STATE {$X_0 \leftarrow \emptyset$, $Y_0 \leftarrow \X$}
\FOR{$i \leftarrow 1$ \textbf{to} $n$}
\STATE $a_i \leftarrow f(X_{i-1} \cup \{u_i\}) - f(X_{i-1})$
\STATE $b_i \leftarrow f(Y_{i-1} \setminus \{u_i\}) - f(Y_{i-1})$
\STATE $a'_i \leftarrow \max\{a_i, 0\}$, $b'_i \leftarrow \max\{b_i, 0\}$
    \IF{$a'_i = 0$ \textbf{and} $b'_i = 0$}
        \STATE $p_i \leftarrow 1$
    \ELSE
        \STATE $p_i \leftarrow a'_i / (a'_i + b'_i)$
    \ENDIF
\STATE \textbf{with} probability $p_i$:
\STATE \hspace{1em} $X_i \leftarrow X_{i-1} \cup \{u_i\}$, $Y_i \leftarrow Y_{i-1}$
\STATE \textbf{else} (with probability $1-p_i$):
\STATE \hspace{1em} $X_i \leftarrow X_{i-1}$, $Y_i \leftarrow Y_{i-1} \setminus \{u_i\}$
\ENDFOR
\RETURN $X_n$
\end{algorithmic}
}
\end{algorithm}
\end{minipage}
\end{center}

\begin{theorem}[Performance of RandomizedUSM~\citep{buchbinder2015tight}]\label{thm:randomized-usm}
Let $f : 2^{\X} \to \real_{\geq 0}$ be a non-negative submodular function. Then $\texttt{RandomizedUSM}$ returns a set $X_n$ satisfying
\[
\mathbb{E}[f(X_n)] \;\geq\; \frac{1}{2}\, f(S^*)
\]
for any $S^* \subseteq \X$. This improves on the $1/3$ guarantee of the deterministic $\texttt{DoubleGreedy}$ algorithm to the tight $1/2$, which is the best factor attainable in the value-oracle model for unconstrained maximization~\citep{feige2011maximizing}; the bound holds in expectation over the algorithm's internal randomness.
\end{theorem}

\subsubsection*{Sample Greedy}

$\texttt{SampleGreedy}$~\citep{harshaw2022power} is a simple randomized algorithm designed for non-monotone submodular maximization subject to a matroid constraint. The key idea is to independently include each element of the ground set with probability $1/2$ to form a random sample $R$, and then run the standard greedy algorithm on $R$ rather than on the full ground set. This random subsampling acts as a form of ``pre-filtering'' that, with constant probability, retains a constant fraction of the optimal solution while discarding elements that would interact poorly with greedy's selections.

\begin{theorem}[Performance of Sample Greedy~\citep{harshaw2022power}]\label{thm:sample-greedy}
Let $f : 2^{\X} \to \real_{\geq 0}$ be a non-negative submodular function, let $\mathcal{I}$ be a matroid constraint, and let $S^*$ be an optimal solution to $\max_{S \in \mathcal{I}} f(S)$. Then $\texttt{SampleGreedy}$ returns a set $S$ satisfying
\[
\mathbb{E}[f(S)] \;\geq\; \frac{1}{4}\, f(S^*).
\]
For $k$ matroids, the guarantee scales as $\frac{k}{(k+1)^2}$. 
\end{theorem}
While this does not match the best approximation factors for monotone objectives, $\texttt{SampleGreedy}$ is significantly simpler to analyze and implement than local search or continuous greedy methods, making it attractive for practical large-scale problems.

\subsection{Algorithmic extensions and variants}
The algorithms presented in this section address the standard offline setting where the full ground set $\mathcal{X}$ and oracle access to $f$ are available throughout. Several important variants relax these assumptions.

Online settings, where items must be irrevocably assigned as they arrive, have been studied extensively for submodular welfare maximization~\citep{lesage2024online, MR3818335}. A related streaming variant~\citep{gomes2010budgeted,badanidiyuru2014streaming} retains only a fraction of a data stream: a threshold-based strategy achieves a provable $(\frac{1}{2}-\epsilon)$ guarantee and has been applied to SLAM-type algorithms~\citep{thorne2024submodular}. Other work considers submodular maximization with limited oracle access~\citep{downie2022submodular} or subject to monotone decreasing set function constraints~\citep{ye2023maximization}, as well as distributed submodular minimization~\citep{testa2018distributed,jaleel2019distributed}. Submodular structure has also been used to reduce the computational cost of chance-constrained programs through a partitioning argument on the constraint set~\citep{frick2019exploiting}.

Extensions to the core definition of submodularity include \emph{adaptive submodularity}~\citep{golovin2011adaptive, maillet2013dynamic, javdani2013efficient, kim2016planning, ellis2024generalized, luan2017fast}, which generalizes diminishing returns to sequential, policy-based selection under uncertainty, and \emph{string submodularity}~\citep{zhang2012submodularity,zhang2013near,zhang2016near,liu2015bounding,van2023improved}, which extends the theory to ordered sequences. The notion of \emph{approximate submodularity} (or weak submodularity) relaxes the diminishing-returns condition and still yields approximation guarantees with degraded constants~\citep{chamon2020approximate, chamon2021approximately, silva2019model}.

\section{Applications}\label{sec:applications}

Table~\ref{tab:applications_taxonomy} summarizes the applications covered in this section, organized by canonical objective, monotonicity of the objective, the most common constraint encountered in that application area, and the worst-case greedy guarantee from Section~\ref{sec:algorithms} that typically applies. The goal is to give the reader a compact map of where the theorems of Section~\ref{sec:algorithms} land in practice before the per-subsection treatments that follow.

\begin{table*}[t]
    \renewcommand{\arraystretch}{1.25}
    \centering
    \footnotesize
    \caption{Taxonomy of submodular structure across the applications of Section~\ref{sec:applications}. ``Monotone?'' indicates whether the canonical objective is monotone; ``Constraint'' lists the most common form. The rightmost column points to the greedy guarantee from Section~\ref{sec:algorithms} that typically applies.}
    \label{tab:applications_taxonomy}
    \begin{tabular}{@{}lllll@{}}
        \toprule
        \textbf{Application area} & \textbf{Canonical objective} & \textbf{Monotone?} & \textbf{Constraint} & \textbf{Typical guarantee} \\
        \midrule
        Sensor scheduling            & $\log\det$ of KF info / Gramian       & often \cmark     & Cardinality, matroid             & Thms~\ref{thm:greedy-cardinality},~\ref{thm:greedy-matroid} \\
        Multi-agent systems          & Coverage, detection probability        & \cmark           & Cardinality, matroid             & Thms~\ref{thm:greedy-cardinality},~\ref{thm:greedy-matroid} \\
        Robust submodular opt.\      & Worst-case-removed / CVaR objective    & \cmark           & Cardinality (with $\tau$-robustness) & Robust variant of Thm~\ref{thm:greedy-cardinality} \\
        Leader--follower systems     & Controllability fraction / influence   & \cmark           & Cardinality                      & Thm~\ref{thm:greedy-cardinality} \\
        Distributed submodular opt.\ & Partitioned submodular function        & \cmark           & Cardinality, partition matroid   & Thms~\ref{thm:greedy-cardinality},~\ref{thm:greedy-matroid} \\
        Game theory                  & Player utility functions               & mixed            & Various                          & Equilibrium existence \\
        System theory                & Gramian log-determinant / rank         & \cmark           & Matroid, knapsack                & Thms~\ref{thm:greedy-matroid},~\ref{thm:greedy-knapsack} \\
        Resource allocation          & Welfare (coverage / utility)           & \cmark           & Cardinality, matroid             & Thms~\ref{thm:greedy-cardinality},~\ref{thm:greedy-matroid} \\
        Social networks / opinions   & Influence spread, consensus error      & \cmark           & Cardinality                      & Thm~\ref{thm:greedy-cardinality} \\
        Informative path planning    & Information gain / mutual information  & \cmark           & Path length, matroid             & Thm~\ref{thm:greedy-matroid} \\
        \bottomrule
    \end{tabular}
\end{table*}

\subsection{Sensor scheduling}\label{subsec:sensor-scheduling}

    A fundamental challenge in networked control and estimation systems is the sensor scheduling problem, which involves strategically selecting a subset of sensors to activate in order to optimize a submodular objective, such as maximizing observability or minimizing overall estimation error in dynamical systems. This notion of diminishing returns is natural, where the marginal benefit of adding a new sensor decreases as the number of already-selected sensors grows. A methodological alternative to the greedy submodular approach is convex relaxation, where the binary selection is relaxed and a log-determinant (D-optimal) objective is optimized \citep{joshi2008sensor}. In most cases, the objective is also monotone: the utility of an additional sensor is non-negative. Consequently, a significant body of research has focused on identifying and exploiting submodularity across a range of sensor scheduling applications, from state estimation and controllability to network security and data collection. A summary of key research directions is presented in Table~\ref{tab:sensor_summary}.

    A core finding in this area is that key performance metrics, which are typically difficult to optimize, exhibit a submodular structure. For instance, \citet{cortesi2014submodularity} and \citet{summers2015submodularity} demonstrated that various functions of the controllability Gramian, such as its log determinant and rank, are submodular. This enables the efficient selection of actuators to optimize system performance, including cardinality-constrained selection subject to minimum control effort \citep{tzoumas2015minimal} (Theorem~\ref{thm:greedy-cardinality}) and actuator placement to ensure structural controllability, which is modeled as a matroid constraint \citep{guo2019actuator, guo2021actuator} (Theorem~\ref{thm:greedy-matroid}). The submodularity property also governs the selection of inputs to achieve synchronization in complex networks \citep{sahabandu2020submodular}.

    Similarly, in state estimation, researchers have established conditions under which objectives related to the Kalman Filter (KF) error covariance matrix are submodular. For linear systems, the log determinant of the KF information matrix is submodular \citep{tzoumas2016sensor}. In \citet{jawaid2015submodularity} a broader set of sufficient conditions for submodularity for objectives involving the KF error covariance is provided. While the trace of the one-step KF error covariance is only weakly submodular \citep{hashemi2020randomized}, a greedy algorithm can still provide provable guarantees of the form $(1-e^{-\gamma})$ via the submodularity ratio $\gamma$, a weak-submodularity analogue of Theorem~\ref{thm:greedy-cardinality}. The situation for the steady-state KF error covariance is more complex, with \citet{zhang2017sensor} providing negative results on its submodularity, even though greedy approaches often perform well in practice. In related work, \citet{hartman2019near} specify conditions under which the mean-squared error (MSE) is supermodular for determining sensor sampling frequencies, similar to \citet{MR4208548} who show conditions for its supermodularity in a scheduling context. \citet{tzoumas2016near} similarly show that the log determinant of the expected squared error in batch state estimation is supermodular.

    Beyond these foundational problems, submodular optimization has been adapted to address more complex and realistic scenarios. The field of robust and resilient optimization leverages submodularity (see Section~\ref{subsec:robust-opt} for an in-depth discussion) to counter adversarial actions, such as when sensors are removed by an adversary \citep{laszka2015resilient}. This has led to the formulation of problems as robust sequential submodular optimization \citep{tzoumas2020robust}. In distributed systems (see Section~\ref{subsec:distributed-opt} for an in-depth discussion), communication constraints and network topology pose significant challenges. \citet{hashemi2018near} use weak submodularity to provide guarantees for a distributed greedy algorithm to minimize estimation error in networks where nodes exchange measurements, while \citet{grimsman2022impact} explore distributed selection using message passing. For large-scale systems, randomized greedy algorithms offer scalability, as demonstrated by \citet{hibbard2023randomized} for sensor selection ($\texttt{StochasticGreedy}$, Theorem~\ref{thm:stochastic-greedy}). The framework also supports the co-design of sensing and control, with \citet{tzoumas2020lqg} showing that LQG control and sensing can be decoupled with an approximately supermodular objective, and \citet{nishida2024sparsity} giving an analogous approximate-supermodularity greedy guarantee for the sparsity-constrained LQR (regulation) problem.
    \citet{abou2022optimizing} also show that distributed sensing in a communication constrained setting can be posed as a submodular optimization problem.

    The versatility of this approach is evidenced by its application across diverse domains and problem types. Submodularity has been used for simultaneous placement and activation of wireless sensors to monitor spatial phenomena \citep{krause2011simultaneous} and for the optimal placement of sensors to increase a robust security index in networked control systems \citep{milovsevic2020actuator}. It also finds use in highly specific applications like sensor selection for biological fractional order systems \citep{tzoumas2018selecting} and for observability in non-linear systems \citep{kazma2023state}. The approach has been extended to problems involving quadratic observation systems \citep{ghasemi2019submodularity} and to determine sensor placement based on a structural observability index \citep{bopardikar2021randomized}. In other areas, submodular optimization has been used for utility maximization in wireless sensor networks \citep{zheng2014submodular}, for antenna selection in beamforming \citep{anevlavis2020beam}, for detecting link failures \citep{rahimian2014detection}, and for sensor placement in traffic networks \citep{mehr2018submodular}. The framework's utility is further shown in applications ranging from detecting binary random variables \citep{hespanha2020optimal} and tracking targets \citep{zhou2019sensor}, to solving problems in robotics \citep{liu2018optimal} and optimal placement for energy storage grids \citep{bucciarelli2020greedy} and battery packs \citep{wolf2012optimizing}. Researchers have also applied submodular optimization to problems with unknown communication channel statistics \citep{wu2019learning} and for sensor configuration in Gaussian processes where the metric is submodular but not directly related to the posterior estimate \citep{laurent2023near}. Finally, \citet{perelman2016sensor} formulate sensor placement for detecting pipe failures as a minimum test cover problem, which reduces to set cover and is thus amenable to greedy submodular optimization, while \citet{ye2019sensor} use the submodularity ratio to provide guarantees for hypothesis testing.

    Across these applications, sensor scheduling problems most commonly carry cardinality or matroid constraints on the number of active sensors, making the greedy algorithm (Theorems~\ref{thm:greedy-cardinality} and~\ref{thm:greedy-matroid}) the primary practical tool. Several open questions are specific to this area. The gap between one-step Kalman filter formulations, for which weak submodularity is established via the submodularity ratio~\citep{hashemi2020randomized}, and the steady-state error covariance, for which standard submodularity is known to fail~\citep{zhang2017sensor}, is not yet bridged; identifying a usable instance-dependent parameter for the steady-state objective would extend formal guarantees to the longer-horizon policies that arise in practice. Co-design of sensing and control has so far been approached through approximate supermodularity of LQG objectives~\citep{tzoumas2020lqg}, with related approximate-supermodularity guarantees for the sparsity-constrained LQR problem~\citep{nishida2024sparsity}, and the conditions under which joint design remains tractable beyond LQG are not well understood. A separate question concerns scaling: online and streaming regimes, in which the schedule must be updated at the system sampling rate rather than recomputed from scratch, call for analogues of $\texttt{StochasticGreedy}$ (Theorem~\ref{thm:stochastic-greedy}) with explicit guarantees under time-varying sensor and noise models.
    
    \begin{table*}[ht!] 
        \scriptsize
        \centering
        \caption{Selected representative work on submodularity in sensor scheduling, grouped by problem domain, system model, objective, and constraint type. The table is not exhaustive; further references are discussed throughout the section.}
        \label{tab:sensor_summary}
        \renewcommand{\arraystretch}{1.3} 
        \resizebox{\textwidth}{!}{%
        \begin{tabular}{@{}llll>{\raggedright\arraybackslash}p{5.5cm}@{}}
            \toprule
            \textbf{Domain} & \textbf{System Model} & \textbf{Objective} & \textbf{Constraints} & \textbf{References}\\
            \midrule
            State Estimation & Linear/Gaussian Process & KF/estimator error covariance & Cardinality & \citep{shamaiah2010greedy,krause2011simultaneous,tzoumas2016near,tzoumas2016sensor,tzoumas2018selecting,jawaid2015submodularity}; \citep{hashemi2020randomized,chamon2020approximate,hartman2019near,hibbard2023randomized,MR4208548}\\
            State Estimation & Linear & KF error covariance & Matroid & \citep{chamon2019matroid,chamon2021approximately}\\
            State Estimation & Linear/Gaussian Process & Resiliency/Robustness & Cardinality & \citep{tzoumas2020robust,laszka2015resilient,de2021bilinear}\\
            State Estimation & Non-linear & Mutual Information & Cardinality & \citep{park2018adaptive,tzoumas2017scheduling}\\
            State Estimation & Non-linear & EKF Error Covariance & Cardinality & \citep{ghasemi2019submodularity}\\
            Distributed Estimation & Linear & KF Error covariance & Cardinality & \citep{hashemi2018near}\\
            Distributed Estimation & Linear & Fisher Information & Partition Matroid & \citep{grimsman2022impact}\\
            Controllability/Observability & Linear & Gramian & Cardinality & \citep{cortesi2014submodularity,summers2015submodularity,bopardikar2021randomized,elsherif2024control}\\
            Actuator Placement & Linear & Controllability metrics & Matroid/Knapsack & \citep{tzoumas2015minimal,guo2019actuator,guo2020actuator,guo2021actuator,kyriakis2020actuator}\\
            Controllability/Observability & Non-linear & Uniform Observability/Variational Gramian & Cardinality & \citep{kazma2023state,kazma2024observability,kazma2024observability2}\\
            \bottomrule
        \end{tabular}
        }
    \end{table*}

\subsection{Multi-agent systems}

Multi-agent systems present significant challenges in coordination and resource allocation, particularly in scenarios such as environmental monitoring, search and rescue, and surveillance. A unifying theme in addressing these challenges is the use of submodular optimization, which provides a powerful framework for modeling and solving complex decision-making problems with provable performance guarantees.

In multi-agent coverage problems, the goal is often to position agents to maximize the detection of events or to gain information about an environment. \citet{sun2019exploiting} demonstrated that the objective for the initial agent location in multi-coverage is a monotone submodular function, a property that is crucial for solving the optimal coverage problem in environments with obstacles (subject to a cardinality constraint on the number of deployed agents, the greedy algorithm provides a $(1-1/e)$ guarantee; see Theorem~\ref{thm:greedy-cardinality}). This idea was extended to the optimal composition of a heterogeneous multi-agent robot team, where \citet{sun2020optimal} exploited the submodularity and curvature of the coverage objective to initialize a continuous optimization algorithm. \citet{welikala2022new} introduced a new curvature measure, called extended greedy curvature, which can be computed in parallel with the execution of a greedy algorithm, providing a performance bound in real-time. In a different take on coverage, the risk-sensitive coverage problem, where agents traverse a graph with risks of failure on edges, was shown to be equivalent to solving a submodular set cover problem aimed at finding the minimum set of paths that ensure robot survival \citep{jorgensen2017risk}.

A closely related area is multi-agent search, where the objective is to maximize the probability of detecting a hidden target. The problem of determining the set of locations that maximize this detection probability, particularly when the probability depends on both the agent and the location, can be formulated as a submodular optimization problem over a matroid \citep{ding2017multi, ding2018moving} (Theorem~\ref{thm:greedy-matroid}). This framework has also been applied to dynamic scenarios, with \citet{hollinger2009efficient} showing that placing searchers at the next time step to search for a moving target is a submodular problem that can be solved via a greedy algorithm. More recent works have addressed the complexities of search with constrained agent motion. \citet{li2024computation} formulated target search as a submodular coverage problem subject to routing constraints, leveraging the properties of the greedy algorithm to provide approximation guarantees. This has been further generalized to a submodular maximization problem with an intersection system constraint \citep{li2024multi}. The search and tracking problem has also been tackled in a decentralized manner, where \citet{miki2018multi} showed how a well-designed submodular reward function can be used in a decentralized sequential greedy algorithm to achieve good results.

While the greedy algorithm is a common solution, its performance can degrade under certain constraints. \citet{shi2021communication} investigated a problem of maximizing a submodular objective in a step-wise fashion while maintaining connectivity in the communication graph for a multi-robot team. They showed that in this setting, a greedy approach can be arbitrarily bad, necessitating the development of a heuristic-based algorithm. In a different vein, a recent direction has been to study distributed motion planning in multi-agent systems with a discrete state space through a submodular minimization lens, providing an alternative to the more common maximization formulations \citep{jaleel2019distributed}.

Further applications in multi-agent systems include multi-robot path planning, where submodular functions quantify sensing quality and yield algorithms robust to adversarial degradation of a robot's sensing ability~\citep{schlotfeldt2021resilient,zhou2022distributed,zhou2023robust}. Risk-aware multi-robot coordination, where the objective is the conditional value at risk (CVaR), has also been formulated as a submodular maximization problem, with sequential greedy algorithms applied to sensor placement, task allocation, and combinatorial auctions~\citep{zhou2022risk}. Submodular objectives arise further in task allocation~\citep{qu2019distributed,johnson2012allowing,prasad2022policies,williams2017decentralized,tihanyi2023multi,segui2015decentralised} and persistent monitoring with heterogeneous agents~\citep{rezazadeh2021sub}.

Multi-agent coverage and search objectives most commonly carry cardinality or matroid constraints, so the greedy algorithm (Theorems~\ref{thm:greedy-cardinality} and~\ref{thm:greedy-matroid}) provides the primary theoretical foundation. Several open questions are particular to the multi-agent setting. When robots can reposition between decisions, the feasible set is no longer a fixed matroid but evolves with the team's trajectory; the routing-constrained coverage formulations of~\citet{li2024computation,li2024multi} treat this in special cases, but a unifying characterization that recovers matroid-style guarantees under motion coupling is open. A second question concerns connectivity-constrained coordination, where greedy can be arbitrarily bad~\citep{shi2021communication}: whether weakly submodular surrogates with provable guarantees exist for the connectivity-preserving variant remains unresolved. Heterogeneous teams, as studied in \citet{sun2020optimal}, introduce a further dependence, since the curvature of the coverage objective depends on the composition of agent capabilities in non-trivial ways, and instance-dependent guarantees that account for this composition would tighten current worst-case bounds in practice.
 
\subsection{Robust submodular optimization} \label{subsec:robust-opt}

Robust submodular optimization addresses the problem of optimizing the selection of a subset of elements from a larger ground set under conditions of uncertainty, failures, or adversarial attacks. This often involves maximizing a submodular set function, frequently assumed to be monotone and normalized, to achieve a desired system performance, but with the critical addition of considering \emph{worst-case scenarios} (e.g., denial-of-service attacks, element deletions, or sensor failures) or incorporating \emph{risk-aware measures} like Conditional Value at Risk (CVaR). The selection process can be a single, static choice, or involve multiple adaptive steps over time, where decisions at each stage account for past events and potential future compromises. Robust or resilient optimization strategies are designed in the literature, aiming to guarantee performance despite unpredictable adverse conditions. These often include constraints such as cardinality limits or more general matroid structures.

For clarity, it is worth detailing one of these constructions here. In the context of the classical cardinality-constrained monotone submodular function maximization problem,~\citet{JBO-ASS-RU:16} introduces a robust formulation against the adversarial removal of elements, building on the earlier robust submodular optimization of~\citet{JMLR:v9:krause08b} (which addressed the worst case over a collection of objectives). This robust formulation, which is central to much of the work in this area, aims to solve
\begin{equation}\label{eq:problem1}
\max_{A\subseteq \X, |A|\leq k} \min_{Z\subseteq A, |Z|\leq \tau} f(A\setminus Z),
\end{equation}
Here, $\X$ represents the ground set, $ f $ is a monotone submodular function over $ \X $,
$k$ is the cardinality constraint on the chosen set $A$, and $\tau$ is the robustness parameter, controlling the maximum number of elements ($Z$) that an adversary can remove from $A$. When $\tau=0$, this robust problem reduces to the standard cardinality-constrained problem.
The results of~\citet{JBO-ASS-RU:16}
provide the first constant factor approximation for the robust formulation considering adversarial removal of up to $\tau$ elements. In particular, among other results, a randomized algorithm is proposed which achieves
$(1 - 1/e) - \epsilon$ approximation for constant $\tau$, and extensions to more complex cardinality constraints are studied. 

Robust submodular maximization has since been used in engineering contexts, and in particular in decision and control. For example,~\citet{tzoumas2017resilient} studies a robust version of submodular maximization where the objective is to find a subset of elements resilient to the worst-case removal of a fraction of its elements, by considering curvature-type properties. The work~\citet{tzoumas2018resilient} extends this idea to the sequential version over a finite horizon (though the results hold for non-submodular set functions too). The works~\citet{hou2019robust} and~\citet{hou2021robust} study robust submodular maximization when the functions are correlated and develop a modified greedy algorithm that uses the \textit{correlation ratio}, which measures the degree to which a single element can simultaneously maximize the marginal gains across a collection of submodular functions, to obtain approximation factors. Finally,~\citet{GS-LZ-PT:23} studies multiple-path planning in adversarial environments by also formulating a minimax version of the submodular maximization problem. Robust submodular optimization is also studied in the context of Conditional-Value-at-Risk, for example in~\citet{zhou2022risk}, with applications to mobility-on-demand, where sequential greedy algorithms are used for maximization problems under uncertainty.

Notable open problems in submodular optimization, particularly concerning robust and resilient variants, revolve around improving approximation guarantees and understanding their fundamental limits. A key challenge is determining if a constant factor approximation is achievable, or if an inapproximability result exists, when the number of adversarial removals ($\tau$) is proportional to the chosen set size ($k$); a non-uniform partitioning approach achieves a constant-factor approximation up to $\tau = o(k)$~\citep{bogunovic2017robust}, leaving the linear regime $\tau \propto k$ open. Furthermore, while a randomized $(1-1/e)-\epsilon$ approximation exists for constant $\tau$ in robust monotone submodular function maximization, it is conjectured that this could be made deterministic~\citep{JBO-ASS-RU:16}.
It is also worth pointing out that execution order can critically affect the performance of greedy methods under limited information, and designing favorable agent/order schedules can yield substantial gains in robustness~\citep{konda2022execution}. Finally, we believe that the potential impacts of robust submodular optimization are not fully utilized in the context of decision and control, particularly in scenarios with more complex physical constraints, e.g., in robotics applications with information constraints, or subject to decentralization. 

\subsection{Leader-follower systems}

In multi-agent systems, a broad class of problems involves leader-follower dynamics, where a designated set of leader nodes dictates the behavior of the remaining follower nodes. In such systems, the followers' states are often modeled as a linear average of their neighbors' states, a linear agreement protocol. Inputs from the leaders propagate through this protocol, influencing the dynamics of the entire system. The central challenge is a combinatorial optimization problem: selecting a subset of leaders to maximize a system-level performance objective, typically subject to a cardinality constraint.

A significant body of research has established that the objective functions for leader selection problems often exhibit a submodular structure, consolidated in the monograph of~\citet{clark2016submodularity}. \citet{clark2011submodular} provided a foundational result, showing that selecting a set of $k$ leaders to maximize system performance in various multi-agent networks is a submodular optimization problem (subject to a cardinality constraint; the greedy algorithm provides the canonical $(1-1/e)$ guarantee, Theorem~\ref{thm:greedy-cardinality}). A subsequent work \citet{clark2012leader} further formalized this by demonstrating that the fraction of nodes in the network that are controllable by a given leader set is a submodular function. This framework has also been applied to specific multi-agent behaviors, such as ensuring the translation and scaling of a swarm formation, which was shown to be a monotone submodular maximization problem \citep{schoof2015efficient}. \citet{mai2018optimizing} addresses how a leader can best influence a network by selecting a limited number of direct followers. The authors' main technical contribution is proving that the objective function measuring the leader's influence is supermodular. 

Conversely, a different class of problems involves minimizing an undesirable system metric, such as error. In a leader-follower setup where leaders act as controllers and followers as the plant, it was shown that the limiting error of the followers' states is a non-increasing supermodular function of the leader set \citep{clark2013supermodular}. Similarly, \citet{clark2014minimizing} demonstrated that an upper bound on the convergence error of the follower agents' states is also a supermodular function. Likewise, for consensus networks with noise-corrupted leaders, the total steady-state variance of the node states is a supermodular function of the leader set, so that the variance reduction obtained by adding leaders is submodular and admits the standard greedy guarantees \citep{mackin2018submodular}.

More advanced leader selection problems incorporate additional complexities of real-world networks. \citet{clark2017submodular} address networks with negative edges and provide conditions under which the constraints in the leader selection problem, modeled as a submodular covering constraint, exhibit submodularity. Furthermore, leader selection to ensure consensus in multi-agent systems with switching topologies, where the network structure changes over time, is addressed by \citet{chen2022leader}.

The leader selection problems surveyed here typically reduce to cardinality-constrained monotone submodular maximization, for which the greedy algorithm achieves a $(1-1/e)$ guarantee (Theorem~\ref{thm:greedy-cardinality}). Several open questions are tied to the structure of the agreement protocol. Most existing results take the network topology and edge weights as given and select only the leader set, but the leader set, edge weights, and protocol gains jointly determine the closed-loop convergence behavior of the followers, and the tractable boundaries of this joint design problem are not well understood. The known positive results center on linear agreement dynamics and small variations of it; whether the corresponding objectives remain (weakly) submodular under nonlinear consensus protocols, communication delays, or actuator saturation is largely open. Finally, the switching-topology results of~\citet{chen2022leader} characterize average-case behavior, but worst-case guarantees under adversarial edge removal or leader compromise, connecting to the framework of Section~\ref{subsec:robust-opt}, would close a gap relative to the static-topology theory.

\subsection{Distributed submodular optimization} \label{subsec:distributed-opt}

In \emph{distributed submodular function maximization}, a group of individuals seeks to maximize a (non-negative) submodular set function over a ground set \(\X\), subject to constraints such as cardinality or partition matroids that limit each individual's access to information about the objective.

The work in this area was originally motivated by decentralization in scenarios with a large ground set. To cope with this,~\citet{mirzasoleiman2016distributed}  proposed \textsc{GreeDi}, a two-stage MapReduce protocol, where data is split across \(m\) workers; each runs $\texttt{Greedy}$ (Section~\ref{sec:algorithms}) locally, the partial solutions are merged, and $\texttt{Greedy}$ is applied once more.  Under mild regularity assumptions, e.g., Lipschitz continuity and decomposability of the objective function, \textsc{GreeDi} nearly matches centralized performance and extends to non-monotone objectives, as well as matroid and knapsack constraints.

Unlike the setting described in~\citet{mirzasoleiman2016distributed}, fully distributed models capture limited information through a directed acyclic \emph{information graph} \(G=(V,E)\)~\citep{gharesifard2017distributed}; acyclicity lets the agents act in a fixed order consistent with \(G\), so that when agent \(i\) acts, its in-neighbors have already committed to their choices. In particular, each agent \(i\) chooses an element \(x_i\in X_i\) based solely on the actions \(X_{\mathrm{in}}(i)=\{x_j\,|\,j\in N(i)\}\) of its in-neighbors \(N(i)=\{j\in V\mid(j,i)\in E\}\), i.e.,
agent $i$ chooses its strategy $x_i$ to maximize its marginal reward relative to its limited information $X_{\mathrm{in}}(i)$:
\[
x_i = \argmax_{x\in X_i}\Delta\big(x\;|\; X_{\mathrm{in}}(i)\big).
\]
A small instance is illustrated in Figure~\ref{fig:seqgreedy}.

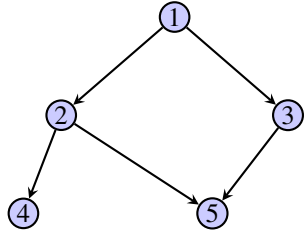
\begin{figure}[htb]
    \centering
    \begin{tikzpicture}[>=stealth, thick]
        \node[main node] (1) at (0,0)      {$1$};
        \node[main node] (2) at (-1.5,-1.3) {$2$};
        \node[main node] (3) at (1.5,-1.3)  {$3$};
        \node[main node] (4) at (-2.0,-2.6) {$4$};
        \node[main node] (5) at (0.5,-2.6)  {$5$};
        \path[draw,->]
            (1) edge (2)
            (1) edge (3)
            (2) edge (4)
            (2) edge (5)
            (3) edge (5);
    \end{tikzpicture}
    \caption{An information graph $G=(V,E)$ for the distributed sequential-greedy framework of~\citet{gharesifard2017distributed}. A directed edge $j \to i$ encodes that agent $j$ is an in-neighbor of $i$; agent $i$ then selects $x_i = \argmax_{x \in X_i} \Delta\big(x \mid X_{\mathrm{in}}(i)\big)$ using only the choices of its in-neighbors. Here $X_{\mathrm{in}}(1) = \emptyset$, $X_{\mathrm{in}}(2) = X_{\mathrm{in}}(3) = \{x_1\}$, $X_{\mathrm{in}}(4) = \{x_2\}$, and $X_{\mathrm{in}}(5) = \{x_2, x_3\}$. The competitive ratio of the resulting scheme is bounded in terms of graph parameters such as the clique number, chromatic number, and (fractional) independence number of $G$.}
    \label{fig:seqgreedy}
\end{figure}

Lower bounds on the performance of greedy algorithms for distributed submodular maximization that depend on the clique number of the information graph, as well as upper bounds in terms of its chromatic number, are provided in~\citet{gharesifard2017distributed}.

It is worth pointing out that the monotonicity assumption plays a key role in the core cardinality-constrained guarantees of both~\citep{mirzasoleiman2016distributed} and~\citep{gharesifard2017distributed}; in the context of the former, more recent results in~\citet{kazemi2021regularized} consider the class of regularized submodular functions, which are in the form of difference between a  monotone submodular function and a modular penalty function; this type of problem is of importance, as the penalty term has applications in data science settings where over-fitting is an issue. 

Back to the setting in~\citet{gharesifard2017distributed}, in subsequent work for example~\citet{grimsman2018impact}, the performance of sequential-greedy schemes is measured by the competitive ratio
\[
\rho(G)=\inf_{f,X}\frac{f\!\bigl(x_{\mathrm{sol}}(f,X,G)\bigr)}{f\!\bigl(x_{\mathrm{opt}}(f,X)\bigr)},
\]
which can be bounded via graph parameters, for example in terms of the (fractional) independence number.  
Subsequent work in~\citep{grimsman2020cost,grimsman2022valid,sun2020distributed,grimsman2020impact,ye2020distributed} has deepened understanding and widened the applicability of distributed greedy strategies. 
The problem of client selection in federated learning settings was studied in \citet{ye2021client} and posed as a submodular optimization problem. Message-passing algorithms, which use consensus-type updates, are studied in~\citet{robey2021optimal}.

Linear-programming-based worst-case constructions reveal how network topology limits greedy performance~\citep{downie2022programming}. This formulation 
allows for an explicit link between graph modifications and worst-case performance; a rather surprising by-product of this is the observation that adding communication edges can decrease greedy's worst-case performance, and related perspectives on altering local incentives/utilities in multiagent optimization appear in~\citet{brown2019feasibility}.

Notions of curvature along with multilinear extension followed by distributed Pipage rounding are used in
~\citep{rezazadeh2023distributed}. 
Using 
higher-order notions such as the supermodularity of conditioning, efficient randomized distributed algorithms, which exploit redundancy and conditioning,  are developed in~\citet{corah2018distributed}.  Communication-cost studies prove that finding minimum-cost information graphs is NP-hard but admit approximations and multi-hop routing via minimum propagation trees~\citep{castiglia2019distributed}.

Optimal parallel execution structures show that partitioning $n$ agents into equally sized batches often suffices, and sparse graphs such as Tur\'{a}n graphs can match dense ones in competitive ratio; incorporating total curvature $\lambda$ tightens these bounds further~\citep{sun2020distributed}. Continuous relaxations through the multilinear extension
support Jacobi-style projected stochastic gradients with a $1/2$-approximation~\citep{du2022jacobi}. Building on such relaxations, privacy-preserving stochastic rounding for distributed submodular optimization trades an optimality factor $\approx 1-(1/e)^{1-\kappa_{\max}\sqrt{1-\beta}}$ for a guaranteed $\beta$-privacy level, where $\kappa_{\max}$ is a curvature constant and $\beta\in[0,1]$~\citep{rezazadeh2022distributed}.

Strategic information sharing, transmitting only the element with highest marginal gain, for instance can improve the bound even when communication is limited~\citep{grimsman2018strategic}. Resource-aware algorithms such as Resource-Aware distributed Greedy for multi-robot teams balance solution quality against strict budgets on computation, memory, and bandwidth by exploiting doubly-submodular objectives and the \emph{centralization-of-information-among-non-neighbors} measure
which quantifies the penalty of decentralization~\citep{xu2022resource}, and distributed leader-selection mechanisms for multi-robot sampling under bandwidth constraints are studied in~\citet{luo2016distributed}. Closely related is the problem of jointly coordinating actions and the underlying information graph, which has been formulated as a self-configurable distributed submodular problem~\citep{xu2024performance}.  

On the minimization side, \citet{testa2018distributed} develop algorithms for distributed submodular minimization in asynchronous, unreliable, and time-varying directed networks.

Developing distributed counterparts of $\texttt{StochasticGreedy}$ (Theorem~\ref{thm:stochastic-greedy}) that achieve scalability without requiring global knowledge of the ground set remains largely unexplored, as does extending distributed guarantees to non-monotone objectives subject to matroid constraints.

\subsection{Game theory}

The central problem at the intersection of game theory and submodular optimization is to understand the collective behavior of multi-player games where individuals make strategic decisions based on personal utility functions, assumed to be submodular. The problem is well motivated, as submodularity has commonalities with convexity and concavity, and in this sense, one hopes that the structure allows one to state existence results for Nash equilibrium, similar to the setting of convex/concave games~\citep{rosen1965existence}. This turns out to be the case, as characterized in~\citet{topkis1979equilibrium}. There is a large body of literature since seeking applications of submodular games in various contexts; while we do not plan to investigate these in detail in this survey, we list a few within the control community. In \citep{MR4094752} submodularity  is used to prove existence of 
Nash equilibrium in a class of 
stochastic games
without relying on Markovian structure, where 
submodularity of the underlying cost functions is essentially used to show that best-reply maps preserve the order of admissible strategies. 
In~\citet{como2021optimal} it is shown that the problem of finding the minimum number of players to target to drive a supermodular game to its greatest Nash equilibrium is NP-complete and proposes an efficient iterative algorithm for its solution. The work~\citet{karaca2017game} characterizes 
 the Nash equilibria of the pay-as-bid electricity market mechanism with submodular utility functions as Pareto-optimal core outcomes. 
In~\citet{karaca2018exploiting}, it is shown how weak supermodularity can be exploited to design coalition-proof mechanisms in some market auctions. 

Another interesting direction has to do with mechanism design, where principled control and intervention strategies are developed to guide the system toward desirable outcomes, e.g.,~\citep{karaca2017game, paccagnan2019utility}. In particular, 
studying the efficiency of an equilibrium has been another direction of interest, 
often quantified through metrics like the Price of Anarchy~\citep{he2013price}. This naturally leads to designing agent-level utility functions that align local incentives with system-level goals~\citep{paccagnan2019utility,paccagnan2021utility}, as well as studying when network-game objectives (e.g., anticoordination) exhibit approximate or average submodularity~\citep{das2022approximate,das2024average}. 

Beyond broadening the range of applications, which have remained relatively limited within the control community, and accounting for the impact of various matroid constraints, a  research direction is to explore refined equilibrium concepts, such as strengthened versions of Nash equilibrium and coarse correlated equilibria, within the framework of submodular games, and more broadly to incorporate modern distributed-learning and behavioral dynamics settings~\citep{liu2022understanding}.

\subsection{System theory}

Minimal controllability/observability problems often consider selecting minimal sets of input nodes (actuators) to guarantee fundamental properties such as controllability, disturbance decoupling, and pole placement  in networked cyber-physical systems~\citep{liu2011controllability,olshevsky2014minimal,MR4339621,pequito2017robust, summers2015submodularity,liu2018partial,tzoumas2015minimal,de2016growing,clark2017toward}. This focus later evolved to incorporate quantitative control energy metrics, such as those derived from the controllability Gramian~\citep{pasqualetti2014controllability}, aiming for more cost-effective control beyond mere system steerability. The mathematical property of submodularity for many of these metrics (e.g., the log-determinant or rank of the Gramian) proved crucial for developing efficient approximation algorithms with theoretical guarantees.

That said, in~\citet{olshevsky2017non}, it was rigorously demonstrated that the average control energy (quantified by the trace of the inverse controllability Gramian) is not necessarily supermodular, directly contradicting earlier claims and highlighting a key complexity in this specific optimization domain; alternative techniques to address such issues have since been explored, for instance in~\citet{siami2020deterministic}. More recent work has further extended these problems to include aspects such as robustness against actuator failures, ensuring partial observability, and optimizing network topology for average controllability in higher-order and temporal networks.
To name a few of the volume of related publications, in \citet{pequito2017robust}, it is shown that a robust minimal controllability problem (determining the minimum number of state variables to be actuated under possible failures) can be cast as a set multi-cover problem. 
Edge protection in bilinear dynamical networks under malicious attack is studied via a submodular formulation in~\citet{de2021bilinear}.
The work~\citet{de2016growing} studies the problem of adding nodes to networks to preserve controllability. The partial observability problem, i.e., how to ensure that a subset of the state in a linear system can be reconstructed when full observability is not possible,
is studied in~\citet{liu2018partial}, where the problem has been recast as a submodular set cover problem. In~\citet{MR4339621}, the design of networks for maximizing  
controllability measures is addressed. Designing topologies for average controllability in networked systems subject to communication constraints is also studied in~\citet{srighakollapu2019optimizing}, and aspects of structural controllability in temporal networks is considered in~\citet{srighakollapu2021optimizing}. 

While several controllability metrics under strong Gramian‐spectral assumptions exhibit \emph{exact} submodularity, some practically relevant objectives fail to do so \citep{summers2019performance}, leaving a web of gaps to close.  Key directions include characterizing weaker notions of sub/supermodularity conditions when considering other relevant forms of energy or potentials, and perhaps proving whether the submodularity ratio can be pinned below a universal constant, or is system–specific. Another interesting avenue of research is to consider curvature‐aware approximation, and establishing whether the forward– and reverse–curvature bounds that now yield guarantees are minimax‐optimal, or whether curvature‐independent (e.g.\ logarithmic) factors are achievable or provably impossible. Obtaining algorithmic lower bounds and proving explicit hardness results are of interest, as they illustrate whether current greedy‐family algorithms are optimal. Finally, structural uncertainties, time‐varying or sampled‐data models, and extending the above questions to settings with matroid‐type structural controllability constraints are possible directions of research. 

\subsection{Resource allocation}
In resource allocation problems, a number of individuals take local decisions to collectively pursue a global objective, often under constraints imposed by limited information or sensing capabilities. The typical approach involves using game-theoretic formulations, where individuals are assigned utility functions designed to guide their self-interested actions. A fundamental concern is then to quantify and improve the efficiency of the emergent equilibria relative to the optimal global outcome, which is commonly measured by the Price of Anarchy. Since 
diminishing returns naturally occur in utility-driven phenomena, submodularity is frequently encountered in such settings and is key to analyzing the efficiency guarantees of emergent stable behaviors. For instance,~\citet{marden2014role,marden2016role} studies the role of informational restrictions on multiagent coordination using submodular objectives. The efficiency of distributed resource allocation problems for the case of submodular welfare functions is studied in~\citet{marden2014generalized}, building on the distributed welfare games of~\citet{JRM-AW:13} and the valid-utility game framework of~\citet{vetta2002nash}, in which any Nash equilibrium of a game with submodular social welfare achieves at least half of the optimal welfare, and~\citet{paccagnan2018importance} studies submodular resource allocation under uncertain or inaccessible information together with the complexity of finding Nash equilibria for submodular objectives in distributed resource allocation. Utility designs that improve the performance of the greedy algorithm in resource allocation games without increasing its runtime are studied in~\citet{konda2024optimal}. Stability and fairness for submodular utility functions in resource allocation is another important topic which is studied in many references, for example~\citet{kyriakis2020stability, marden2012state,paccagnan2019utility,paccagnan2021utility,grimsman2022valid}. 

Submodular structure also pervades resource allocation in energy and communication networks. Energy-storage placement in power networks admits a submodular formulation~\citep{qin2018submodularity} (with related flexibility-scheduling and pricing results for distributed energy resources in~\citet{qin2018automatic}), while strategic scheduling of thermal energy profiles for residential consumers is naturally a non-monotone submodular maximization problem~\citep{albert2015strategic}; cooperative scheduling for microgrids under peak-demand pricing has been studied through the same lens~\citep{valibeygi2019cooperative}. Selecting a minimal set of generators for small-signal stability also admits a submodular formulation~\citep{liu2016mingen}, and a submodular energy-function approach to controlled islanding with provable transient stability is developed in~\citet{cheng2023submodular}. In communications, submodularity has been used in scheduling and channel-allocation problems~\citep{hou2011optimality,xu2023optimal,wang2018optimal} and in multi-sensor transmission-power control over rate-limited channels~\citep{li2019multi}.

Notably, most work in this area is on monotone submodular utility functions, and non-monotone settings are rarely considered; such settings are of importance, for instance, when utility functions are only monotone beyond a certain threshold. Another important open direction is the study of uncertainty in both the utility design and actions taken or unavailable system-level information.

\subsection{Social networks and opinion dynamics}
Submodularity arises in social networks primarily because influence, coverage, and utility functions exhibit natural diminishing returns. 
Maximizing the spread of influence is studied in the classical work~\citep{kempe2003maximizing}, where 
some aspects of the problem are shown to be submodular, and in a volume of papers in the literature on computer and social sciences. 
Closer to decision and control, submodularity has been used, for instance, in leader selection problems for opinion dynamics: in~\citet{yi2021shifting}, leaders are selected in social networks governed by Friedkin--Johnsen dynamics with the goal of optimally shifting follower opinions, with the underlying objective shown to admit submodular surrogates.

Given a networked dynamics, the problem of edge selection, both for activation and protection in adversarial settings, while minimizing resource usage, falls within the context of submodularity. Examples include bilinear formulations in~\citet{chanekar2022encoding} and~\citet{MR4703076}.
The problem of placing resources to influence a social network is also extensively studied, as seen in~\citep{MR4471733,MR4382153,MR4375076,MR4051620,MR3625706,MR4431471,MR4205590,MR4150030,MR4192552}.

Selecting a subset of nodes in consensus networks to minimize ``coherence'', a measure of the variance of deviation from a desired trajectory, is addressed in~\citet{mackin2018second}. In~\citet{rezazadeh2021multi}, distributed gradient strategies are developed for the multilinear extension of constrained submodular maximization problems using a maximum consensus message-passing scheme.

In~\citet{MR4451612}, the problem of controlling segregation in social networks via exogenous incentives is studied by formulating an edge formation problem shown to be submodular.
Other examples related to opinion dynamics include~\citet{MR4413050}, where a susceptible-infected-recovered (SIR) epidemic model is analyzed and the spread of infection is bounded by a supermodular function, which is subsequently minimized to reduce infections. Similarly, in~\citet{MR3580811}, evolutionary dynamics on networks are studied, where submodularity with respect to the mutant set is exploited to analyze evolutionary outcomes.
In~\citet{bini2022graph}, heuristic strategies are proposed for targeting nodes in opinion networks to steer global behavior. Graph-structural heuristics such as degree and closeness centrality are used to identify influential nodes, along with submodularity-based approaches. Notably, it is shown that there are computational limitations to relying solely on greedy strategies.

Beyond influence and opinion dynamics, submodular structure appears across a range of network-analytic problems. Achieving synchronization in networks of phase-coupled oscillators reduces to submodular maximization over partition matroids~\citep{clark2017toward,baggio2024optimal}; optimizing coherence in composite networks~\citep{mackin2017optimizing}, selecting edges for network design~\citep{summers2015submodularity}, and a variety of objectives associated with random walks on graphs~\citep{clark2019structure} are also submodular. Relatedly, the problem of finding the smallest submatrix whose smallest eigenvalue exceeds a threshold can be cast as a submodular set-cover problem~\citep{clark2018maximizing}, with applications to consensus convergence rates and graph design. On the influence side, further submodular formulations appear in relative and fractional influence maximization~\citep{zhao2015relative,umrawal2023fractional}.

Among the promising directions for further investigation are the development of submodular optimization algorithms that remain effective under strategic interactions, potentially through the incorporation of game-theoretic constraints or incentive-compatible mechanisms, and the study of natural objective functions arising in social network settings that are non-monotone and not necessarily submodular, yet admit approximation via submodular or difference-of-submodular surrogates. On this note, and closing with the reference that we started with, we view the latter as particularly important, since many natural influence functions are not submodular in general; see~\citet{kempe2003maximizing} and also~\citet{mossel2010submodularity} and references thereafter.

\subsection{Informative path planning}

Informative Path Planning (IPP) is a class of problems that seeks to find a path that maximizes a specific information-based utility function, such as sensor information gain or uncertainty reduction, subject to constraints on path cost, which could be distance, time, or energy. A key characteristic of IPP is that the utility function is often submodular, meaning it exhibits a diminishing returns property: each new location or observation provides less marginal benefit than the last. This distinguishes IPP from simpler path planning problems that might, for instance, aim to find the shortest path between two points.

Unlike classical submodular optimization problems where the ground set of items is fixed, IPP involves simultaneously determining which vertices to visit and the optimal sequence in which to visit them. This coupled decision-making process significantly increases the problem's complexity, as the solution space is not only a subset of items but a specific permutation of those items that forms a valid path.

IPP problems are broadly categorized into two main types based on the nature of the environment. In continuous space planning, the environment is continuous and the number of potential points to visit is infinite. Algorithms for this problem type often leverage ideas from sampling-based planning algorithms prevalent in robotics. By strategically sampling points in the continuous space, these methods convert the problem into a discrete one that can be more tractably solved. Examples of this approach include uncertainty-driven sampling, as explored by \citep{hollinger2014sampling}, which guides the sampling process toward areas of high information gain. In discrete space planning, this involves a finite set of potential vertices or locations to visit, such as nodes on a grid or a predefined set of landmarks. The problem then becomes one of selecting a subset of these nodes and ordering them optimally. A common approach is to model the problem as a submodular maximization problem with path-related constraints. For instance, some formulations treat the problem as finding a maximum-weight Hamiltonian cycle/path in the graph, where the weight is a submodular function of the cycle edges \citep{jawaid2013maximum,jawaid2015informative}. Other works have framed the TSP with neighborhoods (TSPN) as a submodular maximization problem subject to a spanning tree constraint (which can be cast as a graphic matroid) \citep{clark2019submodular}.

Recent research in IPP has expanded into several advanced directions. For applications such as environmental monitoring or large-scale exploration, multiple robots must coordinate their actions to maximize information gain collectively, introducing challenges around communication and resource allocation. Works such as \citet{corah2019distributed,corah2021volumetric} and \citet{liu2020monitoring} address these challenges by modeling multi-robot constraints using matroids, which provide a powerful framework for handling various forms of resource and independence constraints. A distributed randomized algorithm for submodular maximization with a partition matroid constraint, specifically applied to multi-robot coverage of 3D environments, is proposed in \citet{corah2021volumetric}. In~\citet{jorgensen2024matroid}, the problem of deploying a team of robots under operational considerations is studied, where probabilistic survival constraints are encoded as a matroid; the overall objective, to find a set of paths that maximizes the expected number of sites visited, is then formulated as a submodular problem, enabling constant-factor approximation algorithms.

In dynamic or uncertain environments, the information gain at a location may not be known a priori, and stochastic or online algorithms are used to make decisions in real time. This is often framed using adaptive submodularity, as seen in touch-based localization problems~\citep{golovin2011adaptive,javdani2013efficient}, or through stochastic optimization techniques for discrete planning~\citep{suh2016efficient}. A distinct line of inquiry investigates the long-term behavior of informative paths; for instance, \citet{dressel2018optimality} explores conditions under which an information-gathering trajectory becomes ergodic. With the rise of machine learning, researchers have also applied data-driven approaches to learn the utility functions or planning policies directly from experience~\citep{choudhury2017learning}.

Algorithmic approaches for IPP problems include recursive greedy schemes~\citep{binney2013optimizing} and, more recently, distributed non-monotone submodular optimization techniques relevant for energy-constrained systems~\citep{cai2021non,cai2023energy}. In~\citet{ghaffari2019sampling,corah2021scalable,lim2016adaptive,guo2020actuator}, pruning strategies are developed for settings where the underlying metric is submodular.

Submodularity also appears across a range of adjacent navigation and perception problems. In Simultaneous Localization and Mapping (SLAM), anchor selection can be formulated as a cardinality-constrained submodular maximization problem~\citep{chen2021anchor,kaveti2024oasis}, and submodular structure has also been exploited for feature selection in the SLAM front-end~\citep{jiao2021greedy,zhao2020good}. Closely related uses include informative path planning for 3D reconstruction~\citep{suresh2024greedy} and active perception~\citep{ghasemi2019online}. In multi-robot planning, ensuring trajectory diversity~\citep{arora2015emergency,choudhury2015planner}, reliability-aware coverage planning~\citep{li2021asynchronous}, and decision-theoretic vehicle routing~\citep{shi2023decision} have been formulated through a submodular lens, as has a risk-aware variant of the traveling salesman problem~\citep{balasubramanian2021risk}.

\subsection{Where submodularity fails}\label{subsec:fails}

The submodular structure invoked throughout this section is not universal: several natural objectives in decision and control either provably fail to be sub- or supermodular, or are sub-/supermodular only under restrictive assumptions. Table~\ref{tab:fails} catalogues the most prominent cases reported in the literature, together with the surrogate notion that recovers a usable approximation guarantee, ranging from curvature and the submodularity ratio (Section~\ref{sec:submod-functions}) to approximate supermodularity and the substitution of a different functional of the same operator.

Several of these counterexamples overturned conjectures from the early literature. The claim that the average control energy is supermodular in the actuator set, suggested by~\citet{cortesi2014submodularity}, was shown to be false by~\citet{olshevsky2017non}; the log-determinant and rank of the same controllability Gramian do remain submodular~\citep{summers2015submodularity}, which is why log-det-based actuator placement enjoys greedy guarantees while inverse-Gramian formulations do not. A closely related counterexample is the trace of the rigidity Gramian pseudoinverse~\citep{shames2015rigid,shames2019corrections}. In Kalman filtering, the trace of the one-step error covariance is only weakly submodular, so the standard $(1-1/e)$ bound is replaced by the submodularity-ratio bound $(1-e^{-\gamma})$~\citep{das2011submodular}, here enabled by the weak submodularity of the one-step KF objective established in~\citet{hashemi2020randomized}; the steady-state error covariance fails submodularity even in this weaker sense~\citep{zhang2017sensor}, and no comparable instance-dependent guarantee is currently available. In Markov decision processes, supermodularity of the $Q$-function is sufficient but not necessary for monotone optimal policies, and~\citet{MR4569670} exhibits MDPs in which the optimal policy is monotone yet the $Q$-function is not supermodular. Finally, \citet{summers2016actuator} establishes that LQR, LQG, and dynamic-game actuator-selection objectives are neither sub- nor supermodular in general; \emph{approximate} supermodularity is nevertheless available, together with associated guarantees, in several practically relevant cases: for LQG objectives~\citep{tzoumas2020lqg} and, analogously, for the sparsity-constrained LQR problem~\citep{nishida2024sparsity} (see Section~\ref{subsec:sensor-scheduling}).

\begin{table*}[ht!]
    \scriptsize
    \centering
    \caption{Selected objectives in decision and control for which (super)submodularity provably fails or holds only weakly, together with the surrogate notion that yields a usable approximation guarantee. The list is illustrative rather than exhaustive; cross-references point back to the per-application subsections.}
    \label{tab:fails}
    \renewcommand{\arraystretch}{1.3}
    \begin{tabular}{@{}>{\raggedright\arraybackslash}p{4.4cm}>{\raggedright\arraybackslash}p{3.1cm}>{\raggedright\arraybackslash}p{2.8cm}>{\raggedright\arraybackslash}p{5.4cm}@{}}
        \toprule
        \textbf{Objective} & \textbf{Setting} & \textbf{Status} & \textbf{What salvages it} \\
        \midrule
        Trace of one-step KF error covariance & Sensor scheduling (\S\ref{subsec:sensor-scheduling}) & Weakly submodular & Submodularity ratio $\gamma$: $(1-e^{-\gamma})$ bound for $\texttt{Greedy}$~\citep{hashemi2020randomized} \\
        Trace of steady-state KF error covariance & Sensor scheduling (\S\ref{subsec:sensor-scheduling}) & Not submodular & No general bound known; greedy as a heuristic only~\citep{zhang2017sensor} \\
        Inverse controllability Gramian (avg.\ control energy) & Actuator selection (\S\ref{subsec:sensor-scheduling}) & Not supermodular~\citep{olshevsky2017non} & Substitute another Gramian functional ($\log\det$, rank), which remains submodular~\citep{cortesi2014submodularity,summers2015submodularity} \\
        Trace of rigidity Gramian pseudoinverse & Anchor/sensor placement & Neither sub- nor supermodular & Restricted formulations only~\citep{shames2015rigid,shames2019corrections} \\
        $Q$-function in MDPs & Optimal policy / RL & Not supermodular & Direct verification of policy monotonicity; no general submodular surrogate~\citep{MR4569670} \\
        LQR/LQG/dynamic-game actuator selection & Control co-design (\S\ref{subsec:sensor-scheduling}) & Neither sub- nor supermodular~\citep{summers2016actuator} & Approximate supermodularity of LQG and sparsity-constrained LQR objectives~\citep{tzoumas2020lqg,nishida2024sparsity} \\
        \bottomrule
    \end{tabular}
\end{table*}

\subsection{Further connections}

\emph{DR-submodularity.}
There is little work within the control systems community where the notion of continuous submodularity, introduced in Section~\ref{sec:DR-sub}, is used. In~\citet{sahabandu2019dynamic}, which considers a multi-player dynamic game with imperfect information between a group of defenders and adversaries, the utility functions are shown to satisfy the DR-submodularity property. In recent work~\citep{bunton2022give,bunton2022joint}, a class of problems inspired by regularized optimization, with applications in sparsity-promoting designs and signal denoising, is studied, combining both continuous and discrete components. It is worth pointing out that~\citep{bunton2022joint} only studies the minimization problem; to the best of our knowledge, the maximization problem for such functions has not been studied. We anticipate a growing body of applications of continuous submodular optimization within control theory.

\emph{Reinforcement learning, learning-based decision making, and bandit feedback.}
With the recent reemergence of reinforcement learning as a key tool for decision making in unknown environments, recent work has explored the intersection of submodular optimization with this area. The works~\citet{prajapat2023submodular,de2024global} study Markov decision processes in which the reward is a (monotone) submodular function of the set of states visited along a trajectory, rather than an additive per-step reward, capturing diminishing returns in applications such as coverage control and informative path planning; the objective is to maximize the expected return in an episodic finite-horizon setting. \citet{prajapat2023submodular} propose a policy-gradient method (SubPO) that greedily maximizes marginal gains, while~\citet{de2024global} handle globally defined trajectory rewards that may be submodular, supermodular, or mixed through a submodular semi-gradient scheme that reduces the problem to a sequence of classical reinforcement-learning problems with curvature-dependent guarantees. Under additional assumptions the expected reward becomes DR-submodular, yielding constant-factor approximation guarantees even though the problem is NP-hard to approximate in general. Further work at the intersection of Markov decision processes and submodular optimization includes scalable policy optimization~\citep{sahabandu2021scalable}, hardness and approximability of sensor selection in mixed-observable MDPs~\citep{bhargav2023complexity}, maximizing the size of the reachable state space in factored MDPs~\citep{fiscko2023maximizing}, clustered control of transition-independent MDPs~\citep{fiscko2023clustered}, gradient-bounded dynamic programming for submodular value functions~\citep{lebedev2022gradient}, and a submodular formulation of the multiple-stopping problem in Markov chains that admits an efficient policy~\citep{krishnamurthy2018multiple}. A distinct approach to tracking and monitoring in partially observable environments uses bandit formulations: in~\citep{streeter2008online,xu2023online,xu2023bandit,xu2024leveraging,bedi2020efficient}, repeated submodular maximization under bandit feedback is considered, where at each round the environment presents a new submodular function (possibly adversarially) and the agent must select a subset of actions (e.g., robot assignments) without observing the full function. Bandit-adjacent combinatorial problems with submodular structure also include stochastic probing~\citep{chugg2019submodular} and restless bandits~\citep{akbarzadeh2019restless}.

Submodular structure further arises in a range of learning- and interaction-themed problems at the boundary between control and machine learning: human--robot collaboration~\citep{shi2024inverse}, fleet-scale active learning~\citep{akcin2023fleet}, and information-theoretic optimization of the questions a robot asks a human to reduce its uncertainty~\citep{biyik2022learning}. On the system-identification and data-driven side, submodular maximization has been used for experimental design in Boolean control networks~\citep{busetto2014experimental}, data-driven sampling of trajectories for estimating backward reachable sets of unknown nonlinear systems~\citep{chakrabarty2018data}, and data selection for Bayesian learning~\citep{ye2021near}.

\section{Conclusions}\label{sec:conclusion}

\emph{Identifying submodularity is challenging.} Once an objective has been shown to be submodular and the constraint set fits a familiar structure (cardinality, matroid, intersection of matroids, knapsack), the algorithmic toolbox of Section~\ref{sec:algorithms} applies directly: greedy and its variants match the best attainable factor for cardinality and knapsack constraints and remain the standard scalable choice under matroid constraints (where continuous-relaxation methods close the remaining gap at higher cost), with lazy evaluation and stochastic subsampling available for scale. The harder task in the applications of Section~\ref{sec:applications} is establishing submodularity in the first place, and deciding what to do when it nearly, but not quite, holds.

\emph{A practitioner's quick guide.} For a reader approaching a new selection problem, the choice of algorithm follows the structure of the objective and the constraint. Monotone submodular maximization with a cardinality constraint admits $\texttt{Greedy}$ with the $(1-1/e)$ guarantee of Theorem~\ref{thm:greedy-cardinality}; with a general matroid constraint the same procedure yields the $1/2$ guarantee of Theorem~\ref{thm:greedy-matroid}, and continuous-relaxation methods recover $(1-1/e)$. Non-monotone objectives are best handled by $\texttt{RandomGreedy}$ under a cardinality constraint (Theorem~\ref{thm:random-greedy}, factor $1/e$) and by $\texttt{RandomizedUSM}$, the randomized double greedy, in the unconstrained case (Theorem~\ref{thm:randomized-usm}, factor $1/2$). When the ground set is large, $\texttt{StochasticGreedy}$ removes the dependence of the runtime on the budget $k$ at the cost of an $\epsilon$ loss (Theorem~\ref{thm:stochastic-greedy}), and $\texttt{LazyGreedy}$ delivers substantial practical speed-ups with no loss of guarantee. When monotonicity or submodularity holds only weakly, the same $\texttt{Greedy}$ procedure continues to apply, with the worst-case factor refined by curvature $c$ to $\tfrac{1}{c}(1-e^{-c})$ or by the submodularity ratio $\gamma$ to $(1-e^{-\gamma})$. Table~\ref{tab:applications_taxonomy} maps these tools onto the application areas of Section~\ref{sec:applications}.

\emph{Control-theoretic objectives often sit at the boundary of submodularity.} Many natural objectives in decision and control fail to be exactly sub- or supermodular; Table~\ref{tab:fails} collects several of the prominent cases, including the trace of the one-step Kalman filter error covariance, the average control energy obtained from the controllability Gramian inverse, the $Q$-function of certain Markov decision processes, and the actuator-selection objectives of LQR, LQG, and dynamic games. Several of these counterexamples overturned earlier conjectures, and the appropriate response depends on the case: substituting a different functional of the same operator (for instance the log-determinant or rank in place of the Gramian inverse), invoking curvature or the submodularity ratio to refine the greedy guarantee, or appealing to approximate supermodularity. There is no universally best surrogate, but the family above covers the situations encountered in practice in this survey.

\emph{Distributed greedy depends on the information graph; robust greedy depends on the adversary model.} The gap between centralized and decentralized greedy is governed by parameters of the information graph such as the clique number, chromatic number, and fractional independence number, rather than by the constraint matroid. In robust formulations, the cost of $\tau$-resilient selection is well understood for cardinality removal but largely open under richer adversarial and communication models.

Three cross-cutting open directions stand out from this survey.

\emph{Sharper characterization of control-theoretic objectives.} The gap between one-step and steady-state Kalman filtering~\citep{hashemi2020randomized,zhang2017sensor}, the boundary of the approximate-supermodularity regime for LQG co-design~\citep{tzoumas2020lqg} and sparsity-constrained LQR~\citep{nishida2024sparsity}, and the conditions under which the $Q$-function in MDPs admits a useful submodular surrogate~\citep{MR4569670} are all currently open. Closing these gaps would extend formal guarantees from the cases listed in Table~\ref{tab:fails} to a much broader class of decision and control objectives, including those that arise in nonlinear and time-varying systems, and in policy-conditioned formulations over Markov decision processes.

\emph{Distributed algorithms with provable guarantees under partial knowledge.} The clique- and chromatic-number characterizations of~\citet{gharesifard2017distributed} establish a clear dependence of distributed greedy on the information graph, but tightening these bounds, and developing distributed counterparts of $\texttt{StochasticGreedy}$ and $\texttt{RandomGreedy}$ that retain non-trivial guarantees without global knowledge of the ground set or the objective, is largely open, particularly for non-monotone objectives subject to matroid constraints.

\emph{Robust and sequential maximization beyond cardinality-removal adversaries.} The $\tau$-resilient framework of~\citet{JBO-ASS-RU:16} captures only one mode of adversarial intervention; more realistic models, in which the adversary can alter the underlying information graph, corrupt the value oracle, or exploit time-varying communication patterns, remain largely unexplored, and are increasingly relevant as submodular techniques migrate into safety-critical autonomy stacks. In each of these directions, curvature and approximate-submodularity tools are likely to play a central role.

\section*{Funding}
This work was supported in part by the Natural Sciences and Engineering Research Council of Canada (NSERC) and the Canada Research Chairs Program.

\section*{Declaration of Generative AI Use}
During the final preparation of this work the authors used Generative AI to assist with editorial restructuring, copy-editing, and identification of style inconsistencies. After using this tool, the authors reviewed and edited all content and take full responsibility for the content of the published article.

\end{document}